\documentclass[11pt,reqno]{article}

\usepackage{amssymb}
\usepackage{amsthm}
\usepackage{graphicx}
\usepackage{amsmath}
\usepackage{fancyhdr}
\usepackage{natbib}
\usepackage{multirow}
\usepackage{subfigure}

\usepackage{color}
\definecolor{light-gray}{rgb}{0.6,0.6,0.6}

\bibpunct{[}{]}{;}{n}{,}{,}

\newtheorem{theorem}{Theorem}[section]
\newtheorem{lemma}{Lemma}[section]
\newtheorem{prop}{Proposition}[section]
\newtheorem{assump}{Assumption}[section]

\newcommand{\eqnref}[1]{Eqn.~(\ref{#1})}

\newcommand{\expec}{\mathbb{E}}
\newcommand{\indicator}{1\hspace*{-0.15cm}1}
\DeclareMathOperator*{\argmax}{arg\,max}

\newcommand{\pderiv}[2]{\frac{\partial #1}{\partial #2}}
\newcommand{\pderivs}[2]{\frac{\partial^2 #1}{\partial #2^2}}
\newcommand{\pcderivs}[3]{\frac{\partial^2 #1}{\partial #2\partial #3}}

\theoremstyle{definition}
\newtheorem{define}{Definition}[section]

\newtheorem{remark}{Remark}[section]

\usepackage[top=1in, bottom=1in, left=1in, right=1in]{geometry}

\begin{document}

\title{Dynamic Bertrand Oligopoly}
\author{Andrew Ledvina\thanks{ORFE Department, Princeton University,
    Sherrerd Hall, Princeton NJ 08544; {\em
      aledvina@princeton.edu}. Work partially supported by NSF grant
    DMS-0739195.} \hspace{1cm} Ronnie Sircar\thanks{ORFE Department,
    Princeton University, Sherrerd Hall, Princeton NJ 08544; {\em
      sircar@princeton.edu}. Work partially supported by NSF grant
    DMS-0807440.}}
\date{April 8, 2010; revised June 8, 2010}
\maketitle

\renewcommand{\theenumi}{(\roman{enumi})}
\renewcommand{\labelenumi}{\theenumi}
\renewcommand{\qedsymbol}{$\blacksquare$}

\begin{abstract}
 \small{ We study continuous time Bertrand oligopolies in which a small
  number of firms producing similar goods compete with one another by
  setting prices. We first analyze a static version of this game in
  order to better understand the strategies played in the dynamic
  setting. Within the static game, we characterize the Nash
  equilibrium when there are $N$ players with heterogeneous costs. In
  the dynamic game with uncertain market demand, firms of different
  sizes have different lifetime capacities which deplete over time
  according to the market demand for their good. We setup the
  nonzero-sum stochastic differential game and its associated system
  of HJB partial differential equations in the case of linear demand
  functions.  We characterize certain qualitative features of the game
  using an asymptotic approximation in the limit of small
  competition. The equilibrium of the game is further studied using
  numerical solutions.  We find that consumers benefit the most when a
  market is structured with many firms of the same relative size
  producing highly substitutable goods. However, a large degree of
  substitutability does not always lead to large drops in price, for
  example when two firms have a large difference in their size.}
\end{abstract}

\section{Introduction}
We study competitive markets with a small number of players in which
firms use price as their strategic variable in an uncertain demand
environment. These are known as Bertrand oligopolies. The firms are
selling differentiated but substitutable goods.  Many products, for
instance consumer goods, are sold in markets that fit this structure.
An example might be Pepsi and Coca-Cola in the market for soft
drinks. Oil, coal and natural gas are commodities that can be
substituted for one another for energy production, but which have
different prices per unit of energy produced.

In this paper, we analyze price setting competition in the case of
differentiated goods, and in continuous time under randomly
fluctuating demands. These are nonzero-sum stochastic differential
games that may be characterized by systems of Hamilton-Jacobi-Bellman
PDEs. The large literature on oligopolistic competition deals
primarily with the static problem, and we refer to \citet{friedman83}
and \citet{vives99} for background and references.
\citet{cournot1838} provided the first analysis of an oligopoly where
firms' strategic interactions are taken into account. He assumed that
firms compete using \textit{quantity} as their strategic variable and
then take prices as determined by the market through an inverse demand
function, that is, a mapping from quantity to price. In a scathing
review of Cournot's paper, \citet{bertrand1883} argued that firms
compete using \textit{price} as their strategic variable and then
produce to clear the market demand arising from a demand function,
that is, a mapping from price to quantity. In actuality, some markets
may be better modeled as Cournot, and others as Bertrand, and we do
not enter that debate here.

In both these original models, however, the goods were homogeneous,
that is perfectly substitutable. This means that the only difference
between the firms is the price they set or the quantity they
produce. In the price setting game, this induces the behavior that,
if the prices between two goods are not equal, consumers will only
purchase the lower priced good. Here, we do not assume goods are
perfectly substitutable, and therefore, if firms have prices that
differ, they may all still receive some demand from the market.
Models of product differentiation originated with
\citet{hotelling29} and \citet{chamberlin33}, and were extended by
\citet{daspremont79}, among others. See \citet[Chapter
3]{friedman83} for an excellent discussion.

Most of the continuous-time models are in a linear-quadratic (LQ)
set-up, which has convenient analytical properties. We refer to
\citet{engwerda05} for details, and \citet{hamadene98} for an approach
via BSDEs.  \citet{vives04} work within the context of a differential
game of a duopoly with differentiated products. Their general model
allows for either Cournot or Bertrand competition in which the players
control the rate of change of the rate of production or price,
respectively, in the LQ setting.

There has been much recent interest in other types of stochastic
differential games. We mention, for example, \citet{lasryLions07jap},
who consider Mean Field Games in which there are a large number of
players and competition is felt only through an average of one's
competitors, with each player's impact on the average being
negligible. \citet{bensoussan09} study leader-follower differential
games in real options problems. \citet{ekeland08} and \citet{bjork08}
analyze time-inconsistent control problems which can be viewed as
games against one's future self. Energy markets in which a small
number of firms control supply can be viewed as Cournot oligopolies
and they are analyzed in the context of exhaustible resources in
\citet{exhres}.

The state variable of the firms in our model is their remaining lifetime capacity. This is a quantity
whose value at time zero represents all of the possible production a firm can undertake over its
lifetime before it goes out of business. The primary reason for this choice is that it is a
natural way to capture the notion of relative sizes of firms. A firm with a very large lifetime capacity
is a major market participant whose decisions greatly affect the prevailing price in the market. A
firm with a very small lifetime capacity has very little market power. Here, we analyze the effect
that participant size has on Bertrand markets over time. For producers
of consumer goods, the lifetime capacities could be proxied by past
volume of sales projected forward into the expected lifetime of the
firm. In the case of exhaustible resources, the state variables would
be estimates of remaining oil, coal or natural gas reserves, for
instance.

In Section \ref{sec:StaticGame}, we set up the static version of our
price setting game and prove the existence and uniqueness of a Nash
equilibrium.  The resulting equilibrium price functions are inputs for
Section \ref{sec:diffGame} where we present the dynamic
Bertrand game. We characterize the price strategies of the firms using
the solution to a system of coupled nonlinear PDEs. We analyze in
detail the problem of a duopoly with linear demand functions and the
related monopoly problem. Analytically, we obtain an asymptotic
expansion in powers of a parameter that represents the extent of
competition between the firms in a deterministic game. In Section
\ref{sec:numericalAnalysis}, we present the numerical solution of our
system of PDEs that allows us to characterize the price strategies and
resulting demands of firms in the stochastic game. Finally, in Section
\ref{sec:conclusion}, we conclude and discuss further lines of
research.

\section{Static Bertrand Game}\label{sec:StaticGame}
The main purpose of this paper is to analyze a dynamic price setting
game in continuous time. In order to fully understand this game, we
first analyze a static version of the game and prove existence and
uniqueness of a Nash equilibrium. This is used in establishing the
system of HJB PDEs of the dynamic game in the next section.

We assume a market with $N$ firms where each firm uses price as a strategic variable in noncooperative competition with the remaining firms.
Associated to each firm $i\in\left\{1,\ldots,N\right\}$ is a variable
$p_i\in\mathbb{R}_+$ that represents the price at which firm $i$
offers its good for sale to the market. We denote by $p$ the vector of
prices whose $i$th element is $p_i$.

\subsection{Systems of Demand}\label{sec:systemsOfDemand}
Given prices of the firms, we specify the resulting market demands
for each firm's good. For each firm $i\in\left\{1,\ldots,N\right\}$,
there exists a demand function
$D_i^N\left(p_1,p_2,\ldots,p_N\right): \mathbb{R}^N_+ \rightarrow
\mathbb{R}$. We first state some natural properties of these demand
functions.

\begin{assump}[Properties of Demand Functions]\label{assump:demandProp}
For all $i=1,\ldots,N$, $D^N_i$ is smooth in all variables, and
\begin{equation*}
D^N_N(0,\ldots,0) > 0, \qquad \pderiv{D^N_i}{p_i} < 0,\qquad \mbox{and}\qquad \pderiv{D^N_i}{p_j} > 0 \textrm{ for } i\neq j.
\end{equation*}
\end{assump}
We further assume that firms are distinguished only by the prices they set.
\begin{assump}[Exchangeability of Firms]\label{assump:sym}
For fixed $p_1,\ldots,p_N$ and all $i,j\in\left\{1,\ldots,N\right\}$,
\begin{equation*}
D^N_i\left(p_1,\ldots,p_i,\ldots,p_j,\ldots,p_N\right) = D_j^N\left(p_1,\ldots,p_j,\ldots,p_i,\ldots,p_N\right).
\end{equation*}
This implies that the demand function is invariant under permutations of the other firms' prices. That is, for any $j,k\in\left\{1,\ldots,N\right\}\backslash\left\{i\right\}$, we have
\begin{equation*}
D^N_i\left(p_1,\ldots,p_i,\ldots,p_j,\ldots,p_k\ldots,p_N\right) = D^N_i\left(p_1,\ldots,p_i,\ldots,p_k,\ldots,p_j\ldots,p_N\right).
\end{equation*}
\end{assump}
Smoothness of these functions is for convenience, and, naturally,
the market has positive demand if the prices are low enough. The key
assumption in the above is that demand for an individual firm is
decreasing in the firm's own price and increasing in the price of
their rivals. This assumption implies that we only deal with
substitute goods. A classical example is Coca-Cola versus Pepsi.
However, such goods can also be of different kinds and thus not
directly replaceable, yet they still exhibit substitutability. For
example, an iPod and a compact disc. One cannot directly replace the
other, but we expect a drop in the price of iPods to cause a drop in
the demand for compact discs. Contrary to this type of good, there
are goods known as complementary goods, such as hot dogs and hot dog
buns, but we do not consider those kinds of competition here.

We now make additional convenient assumptions.
\begin{assump}[Finite Choke Price]\label{assump:chokePriceN} Fix a firm $i\in\left\{1,\ldots,N\right\}$. For any fixed set of prices $p_{-i} \triangleq \left(p_1,\ldots,p_{i-1},p_{i+1},\ldots,p_N\right)$, we assume there exists a ``choke price", $\hat{p}_{i}\left(p_{-i}\right) < \infty$, such that
\begin{equation}\label{eqn:chokePriceDef}
D^{N}_{i}\left(p_1,\ldots,p_{i-1},\hat{p}_i,p_{i+1},\ldots,p_{N}\right) = 0.
\end{equation}
\end{assump}
Note that this ``choke price" is unique by Assumption \ref{assump:demandProp} because $\partial D^N_i/\partial p_i < 0$. This price is also positive by the same assumption and Assumption \ref{assump:sym} because $D^N_N\left(0,\ldots,0\right) = D^N_i\left(0,\ldots,0\right) > D^N_i\left(0,\ldots,0,\hat{p}\left(0\right),0,\ldots,0\right) = 0$. This implies $\hat{p}_i\left(0\right) > 0$.

\begin{remark}\label{remark:additive}For example, suppose each firm's demand depends on its rivals' prices only through their \textit{sum}: $D^N_i = f\left(p_i,\sum_{j\neq i}p_j\right)$ where $f(x,y):\mathbb{R}_+\times\mathbb{R}_+ \rightarrow \mathbb{R}$ is a smooth function which is increasing in $y$, decreasing in $x$, and such that there exists a solution $x$ to $f(x,y) = 0$ for every $y$. Then, it is easy to see that this demand system satisfies Assumptions \ref{assump:demandProp}, \ref{assump:sym} and \ref{assump:chokePriceN}.
\end{remark}

The actual demand that each firm faces cannot be negative as the firms are suppliers. For a fixed price vector $p$, we define $D_i(p)$, without the superscript, as the \textit{actual} demand firm $i$ receives in the market. Suppose first, for simplicity, that $p_1\leq p_2\leq\cdots\leq p_N$.
If this is not the case, then we can re-order the firms, carry out the following procedure, and then return them to their original order once their demands have been determined.
We next show that if prices are ordered, then this same ordering carries over to the demands.
\begin{prop}[Price order implies demand order]\label{prop:pOrderImpDOrderN}
Fix a vector of prices $p$. Suppose they are ordered such that $p_1\leq p_2\leq \cdots\leq p_N$. Then,
\begin{equation*}
D^N_1\left(p_1,\ldots,p_N\right) \geq D^N_2\left(p_1,\ldots,p_N\right) \geq\cdots\geq D^N_N\left(p_1,\ldots,p_N\right).
\end{equation*}
\end{prop}
\begin{proof}
Using the ordering of the prices, the properties of the derivatives of the demand functions, and Assumption \ref{assump:sym}, we have
\begin{eqnarray*}
D^N_N\left(p_1,\ldots,p_{N-2},p_{N-1},p_N\right) &=& D^N_{N-1}\left(p_1,\ldots,p_{N-2},p_N,p_{N-1}\right) \\
&\leq& D^N_{N-1}\left(p_1,\ldots,p_{N-2},p_{N-1},p_{N-1}\right) \\
&\leq& D^N_{N-1}\left(p_1,\ldots,p_{N-2},p_{N-1},p_N\right).
\end{eqnarray*}
The result then follows for all $D^N_i$ by applying the same procedure.
\end{proof}
To determine the actual demand $D_i(p)$, we begin with the demand function $D^N_N(p_1,p_2,\ldots,p_N)$. If $D^N_N(p) \geq 0$, then it must be the case that $D^N_i(p) \geq 0$ for all $i=1,\ldots,N$ by Proposition \ref{prop:pOrderImpDOrderN}. The actual demand that each firm faces is $D^N_i(p)$. Hence,
\begin{equation*}
D_i(p) = D^N_i(p) ~~\textrm{ for all } i=1,\ldots, N,
\end{equation*}
and all the demands are determined. Otherwise, if $D^N_N(p) < 0$, then the price $p_N$ is too high relative to the preference structure of the market to make any sales in the market. Thus, this firm will receive no demand from the market and the demand of the remaining firms must reflect this fact. The demand for firm $N$ is set to zero, $D_N(p) = 0$, and the demand for the remaining firms is determined by considering their residual demand functions, which we now define.

Let us consider the general case of residual demand for $n$ firms when all firms $i > n$ receive zero demand at the current set of prices.

\begin{define}[Consistency of Demand]\label{define:consist} For each $n\in[1,N-1]$, we define the $n$-firm system of demand functions from the $(n+1)$-firm system of demand functions through
\begin{equation*}
D_i^n\left(p_1,p_2,\ldots,p_n\right) = D_i^{n+1}\left(p_1,p_2,\ldots,p_n,\hat{p}_{n+1}\right) \hspace{0.5cm} \textrm{ for } i=1,\ldots,n,
\end{equation*}
where $\hat{p}_{n+1}\left(p_1,\ldots,p_n\right)$ is defined by $D^{n+1}_{n+1}\left(p_1,\ldots,p_n,\hat{p}_{n+1}\right)=0$.
\end{define}

Definition \ref{define:consist} encapsulates that, if firm $n+1$ sets price so high that $D^{n+1}_{n+1}(p) < 0$, demands for firms $i<n+1$ are consistently adjusted as if firm $n+1$ set price $\hat{p}_{n+1}$ that realizes it exactly zero demand. This will give rise to actual demands which are continuous as a firm raises its price through the level at which it receives zero demand, and the market effectively has one less player. We shall see in Remark \ref{remark:genDemandSystems} that a common way of generating demand systems through a representative consumer's utility maximization problem yields a demand system with this consistency property.

In Definition \ref{define:consist}, we do not necessarily know that $\hat{p}_{n+1}$ exists for all $n$. We shall prove that they do exist in Proposition \ref{prop:existenceOfChokePrice}, but first we need to make one additional assumption.
\begin{assump}\label{assump:downwardSlopingDemand}
For all $n\in[1,N-1]$ and all $i=1,\ldots,n$, we assume
$\pderiv{D^n_i}{p_i} < 0$.
\end{assump}
We will see in Propositions \ref{prop:existenceOfChokePrice} and \ref{prop:inherit} that the other natural properties of Assumptions \ref{assump:demandProp}, \ref{assump:sym} and \ref{assump:chokePriceN} are inherited by the lower level demand functions. However, it is necessary to assume that lower level demand functions are decreasing in each player's own price.

\begin{prop}\label{prop:existenceOfChokePrice}For each $n\in[1,N-1]$, there exists a finite ``choke price" $\hat{p}_{n+1}(p_1,\ldots,p_n)$ where
\begin{equation*}
D^{n+1}_{n+1}\left(p_1,\ldots,p_n,\hat{p}_{n+1}\right) = 0.
\end{equation*}
\end{prop}
\begin{proof}
See Appendix \ref{appendix:proof}.
\end{proof}

\begin{prop}[Inherited Demand Properties]\label{prop:inherit}
The functions $D_i^n$, as defined in Definition \ref{define:consist}, are smooth in all variables, and
$\pderiv{D^n_i}{p_j} > 0 \textrm{ for } i\neq j$, and $i=1,\ldots,n$.

The functions $D_i^n$ also inherit the symmetry of the functions $D_i^N$ from Assumption \ref{assump:sym}. Furthermore, they inherit the ordering of demand shown in Proposition \ref{prop:pOrderImpDOrderN}. That is, for a fixed vector of prices $p$, ordered such that $p_1\leq p_2\leq \cdots\leq p_N$, we have for any $1\leq n\leq N$
\begin{equation*}
D^n_1\left(p_1,\ldots,p_n\right) \geq D^n_2\left(p_1,\ldots,p_n\right) \geq\cdots\geq D^n_n\left(p_1,\ldots,p_n\right).
\end{equation*}
\end{prop}
\begin{proof}
The smoothness of the functions $D_i^n$ clearly comes directly from
that of $D^N_i$. Furthermore, the symmetry of these functions is
also clearly inherited. To show the positive transverse derivative,
we simply compute. We show only at the level $N-1$; for all
$n\in[1,N-1]$, the property will follow in the exact same way from
the function at the level $n+1$ using Assumption
\ref{assump:downwardSlopingDemand}. We first take the derivative of
\eqnref{eqn:chokePriceDef} with respect to $p_j$ for $j\neq N$,
which gives
\begin{equation*}
\pderiv{\hat{p}_N}{p_j} = -\frac{\partial D^N_N /\partial p_j}{\partial D^N_N/\partial p_N} > 0, \mbox{ because }\pderiv{D^N_N}{p_j} > 0 \mbox{ and } \pderiv{D^N_N}{p_N} < 0.
\end{equation*}

Then, for $i=1,\ldots,N-1$ and $j\in\left\{1,\ldots,N-1\right\}\backslash\left\{i\right\}$, we have
\begin{equation*}
\pderiv{D^{N-1}_i}{p_j} = \pderiv{D^N_i}{p_j} + \pderiv{D^N_i}{p_N}\pderiv{\hat{p}_{N}}{p_j}.
\end{equation*}
This is positive because of Assumption \ref{assump:demandProp} and because $\pderiv{\hat{p}_N}{p_j} > 0$ for any $j\neq N$. With these properties, the ordering of demands follows by applying the exact same proof as Proposition \ref{prop:pOrderImpDOrderN}.
\end{proof}

We can now specify completely the demand function that each firm faces for a given set of prices.

\begin{define}\label{actdem}(Actual Demands) Given an ordered price vector $p$:
\begin{itemize}
\item If $D_N^N(p) \geq 0$, then $D_i(p) = D^N_i(p)$ for all $i=1,\ldots,N$.
\item Otherwise, find $n\in\left\{1,\ldots,N-1\right\}$ such that
\begin{equation*}
D^{n+1}_{n+1}\left(p_1,\ldots,p_n,p_{n+1}\right) < 0,
\qquad \textrm{ and } \qquad
D^{n}_n\left(p_1,\ldots,p_n\right) \geq 0.
\end{equation*}
For such an $n$, the actual demands of firms $n+1,\ldots,N$ are equal to zero, and $D_i^n$ give the actual demands for each firm $i\in\left\{1,\ldots,n\right\}$:
\begin{equation*}
D_i(p) = \left\{\begin{array}{ccl} D^n_i\left(p_1,p_2,\ldots,p_n\right) & & \textrm{ for } i=1,\ldots,n \\
0 & &\textrm{ for } i=n+1,\ldots,N\end{array}\right.
\end{equation*}
\item If no such $n$ exists, then $D_i(p) = 0$ for all $i=1,\ldots,N$.
\end{itemize}
\end{define}

\subsection{Example: Linear Demand}\label{sec:linDemandDef}
We present demand functions that are affine in the prices of all
firms. This is the demand structure we will use in the dynamic game of
the following sections. For fixed $N$, we start with positive
parameters $A, B, C$ such that $B >(N-1)C$.  This latter
condition on the parameters will be justified in what follows. With
these parameters, we define
\begin{equation}\label{eq:dNDefine}
D_i^N\left(p_1,\ldots,p_N\right) \triangleq A - Bp_i + C\sum_{j\neq i} p_j, ~~\textrm{ for } i=1,\ldots,N.
\end{equation}
Notice that $D^N_N(0,\ldots,0) = A > 0$, and $D_i^N$ is of the form given in Remark \ref{remark:additive}. Therefore the demand functions satisfy Assumptions \ref{assump:demandProp}, \ref{assump:sym} and \ref{assump:chokePriceN}.

\begin{prop}\label{prop:linDemand}For each $n\in[1,N-1]$, we have
\begin{equation}\label{eq:dnConsistDefine}
D_i^n\left(p_1,\ldots,p_n\right) = a_n - b_np_i + c_n\sum_{\underset{j\leq n}{j\neq i}}p_j, ~~\textrm{ for } i=1,\ldots,n,
\end{equation}
where, for $2\leq n\leq N$,
\begin{equation}
a_{n-1} = a_n\left(1+\frac{c_n}{b_n}\right), \qquad b_{n-1} = b_n\left(1 - \frac{c_n^2}{b_n^2}\right), \qquad c_{n-1} = c_n\left(1 + \frac{c_n}{b_n}\right),\label{recursiveABC}
\end{equation}
with $a_N = A$, $b_N = B$ and $c_N = C$.
\end{prop}
\begin{proof}
Using Definition \ref{define:consist}, we solve for the choke price $\hat{p}_N$ by setting $D_N^N$ in \eqnref{eq:dNDefine} to zero and solving for $p_N$. This results in
$\hat{p}_{N} = B^{-1}\left(A + C\sum_{i=1}^{N-1} p_i\right)$.
Substituting into \eqnref{eq:dNDefine} we obtain, for $i=1,\ldots,N-1$,
\begin{equation*}
D_i^{N-1}(p_1,\ldots,p_{N-1}) = A\left(1 + \frac{C}{B}\right) - B\left(1 - \frac{C^2}{B^2}\right)p_i + C\left(1+\frac{C}{B}\right)\sum_{\underset{j\leq N-1}{j\neq i}}p_j,
\end{equation*}
which establishes \eqnref{recursiveABC} for $n=N$.
We can repeat this procedure for $1\leq n\leq N-1$, where we find
$\hat{p}_{n+1} = b_{n+1}^{-1}\left(a_{n+1} + c_{n+1}\sum_{i=1}^n p_i\right)$,
and this results in the demand system at the level $n$ being given by \eqnref{eq:dnConsistDefine} with the recursively defined parameters given in \eqnref{recursiveABC}.
\end{proof}

\begin{prop}\label{prop:recursionSoln}The explicit solution of the recursion (\ref{recursiveABC}) is given by
\begin{equation}
a_n = \frac{\alpha}{\beta+(n-1)\gamma}, \qquad b_n = \frac{\beta+(n-2)\gamma}{(\beta+(n-1)\gamma)(\beta-\gamma)}, \qquad c_n = \frac{\gamma}{(\beta+(n-1)\gamma)(\beta-\gamma)},\label{ABCintermsOfOthers}
\end{equation}
where we define
\begin{equation}
\gamma =
\frac{C}{\left(B-(N-1)C\right)\left(B+C\right)}, \qquad \alpha = \gamma\cdot A\cdot\left(\frac{B}{C}+1\right), \qquad \beta =
\gamma\cdot\left(\frac{B}{C}-(N-2)\right). \label{alphaSoln}
\end{equation}
\end{prop}
\begin{proof}
Simple algebra shows that the expressions in \eqnref{ABCintermsOfOthers} satisfy the
recursions in \eqnref{recursiveABC}. All that remains to show is
that $a_n,b_n,c_n$ are positive and well-defined. By examination of
\eqnref{ABCintermsOfOthers}, this will be the case provided
$\alpha,\beta,\gamma$ are positive, and $\beta>\gamma$ because of the
denominator in the last two expressions in
\eqnref{ABCintermsOfOthers}. We see from the first expression in
\eqnref{alphaSoln} that $\alpha>0$ if $\gamma>0$. Furthermore, we have
that $\beta>\gamma$, and therefore $\beta>0$ if $\gamma>0$ and if $B
>(N-1)C$. This is exactly the condition we assumed above on the
parameters $B$ and $C$. Therefore, we need only show that $\gamma>0$,
but this again will be true if $B  >(N-1)C$. Hence, we have that
$a_n,b_n,c_n$ are positive for all $n$.\end{proof}

\begin{remark}
Note that Assumption \ref{assump:downwardSlopingDemand} is satisfied by the demand functions in \eqnref{eq:dnConsistDefine} because $\pderiv{D^n_i}{p_i} = -b_n < 0$, for all $n\leq N$ and $i\leq n$.
\end{remark}

\begin{remark}[Generating Demand Systems]\label{remark:genDemandSystems}
  One can generate demand systems that satisfy our assumptions by
  starting with a utility function and using the utility maximization
  problem of a representative consumer. Let $U(q):\mathbb{R}^N_+
  \rightarrow \mathbb{R}$ be a smooth and strictly concave utility
  function, where $q$ is a vector representing quantities of the
  different products. We assume that a representative consumer solves
  the problem of maximizing utility of consumption minus the cost of
  that consumption: $\max_q U(q) - pq$.  One then obtains inverse
  demands from the first order conditions of this maximization problem
  $p = \nabla U$. The Jacobian of the inverse demand system equals the
  Hessian of $U$, which implies the system is invertible. We obtain
  the direct demand system $\{D_i^N\}$ by inverting this system for quantity as a
  function of price. Concavity of $U$ implies $\partial D_i^N/\partial p_i
  < 0$, but we cannot tell directly from $U$ if $\partial D_i^N/\partial
  p_j \geq 0$ holds for $j\neq i$. Therefore, an additional assumption
  must be made at the level of the demand functions in order to model
  substitute goods. However, the consistency property  in Definition
  \ref{define:consist} is guaranteed.  Suppose firm $N$ is removed from the utility
  function, then the demand functions $\{D_i^{N-1}\} $ derived this
  way, inverting $\nabla U$ after setting $q_N=0$, are consistent with
  the $\{D_i^N\}$, and similarly for the lower level demand functions
  $\{D_i^n\}$.
The linear system of demand introduced in Section
  \ref{sec:linDemandDef} can be obtained by using the quadratic utility function
\begin{equation*}
U(q) = \alpha \sum_{i=1}^Nq_i -
\frac{1}{2}\left(\beta\sum_{i=1}^Nq_i^2 + \gamma\underset{i\neq
    j}{\sum\sum}q_iq_j\right).
\end{equation*}
While it is not necessary to assume that demand is derived from
utility, it can be shown that, under some mild conditions, given a
system of demand, there exist preferences that rationalize that
demand, which result in a utility function consistent with that demand
system. See \citet[Section 3.H]{mwg95} for more details.
\end{remark}

\subsection{Nash Equilibrium}\label{subsec:StaticGame:NE}
We now analyze the static Bertrand game. Each firm
$i\in\left\{1,\ldots,N\right\}$ has an associated constant marginal
cost, denoted by $s_i$. We denote by $s$ the vector of costs with
$i$th element equal to $s_i$. Each firm chooses its price to maximize
profit in a non-cooperative manner, but they must do so while taking
into account the actions of all other firms. Firms
choose  prices to maximize profit in the sense of Nash equilibrium.
The profit function $\Pi_i: \mathbb{R}^N\times\mathbb{R}
\rightarrow\mathbb{R}_+$ for firm $i$ is given by
\begin{equation}\label{eq:profitFunction}
\Pi_i \left(p_1,p_2,\ldots,p_N,s_i\right) \triangleq D_i(p)\cdot\left(p_i - s_i\right),
\end{equation}
where $D_i(p)$ was defined using the procedure in Section \ref{sec:systemsOfDemand}.

In order to simplify exposition, we assume, possibly after a suitable
relabeling, that firms are ordered by costs:
$0 \leq s_1 \leq s_2 \leq \cdots \leq s_N$.
\begin{define}\label{define:NE}
A vector of prices $p^{\star} = \left(p_1^{\star},p_2^{\star},\ldots,p_N^{\star}\right)$, is a
\textbf{Nash equilibrium} of the Bertrand game if
\begin{equation}\label{eq:NEaddition}
p_i^{\star} = s_i \textrm{ whenever } D_i\left(p^{\star}\right)= 0,
\end{equation}
and
\begin{equation}
p_i^{\star}  = \argmax_{p\geq s_i}
\Pi_i\left(p_1^{\star},p_2^{\star},\ldots,p_{i-1}^{\star},p,p_{i+1}^{\star},\ldots,p_N^{\star},s_i\right)
\label{staticBertrandNashEqDefMaxCondition}
\end{equation}
for all $i=1,\ldots,N$.
\end{define}
\eqnref{staticBertrandNashEqDefMaxCondition} says the Nash equilibrium
is a fixed point of best-responses.  \eqnref{eq:NEaddition} says that
whenever a firm receives zero demand in equilibrium, it sets its price
equal to cost, which is a best-response, meaning it satisfies
\eqnref{staticBertrandNashEqDefMaxCondition}. This makes the
best-response $p_i^\star$ a well-defined function.

In the game with heterogeneous costs, some firms may receive zero
demand, and so we first consider subgames which, for $n=1,\ldots,
N$, involve only the first $n$ players. Let $p^{\star,n} =
\left(p_1^{\star,n},\ldots,p_n^{\star,n},\right.$
$\left.s_{n+1},\ldots, s_{N}\right)$, where the first $n$ components
solve the Nash equilibrium problem with profit functions
$\Pi^n_i(p_1,\ldots,p_n) = D_i^n(p_1,\ldots,p_n)\cdot(p_i-s_i)$ and
$D_i^n$, introduced in Definition \ref{define:consist}, is the
demand function for the $n$-player game. In other words,
\begin{equation}
  p_i^{\star,n}  = \argmax_{p\geq 0}
  \Pi_i^n\left(p_1^{\star,n},p_2^{\star,n},\ldots,p_{i-1}^{\star,n},p,p_{i+1}^{\star,n},\ldots,p_n^{\star,n}\right),
  \qquad i=1,\ldots,n.
\label{substaticBertrandNashEqDefMaxCondition}
\end{equation}
\begin{assump}
  We assume that, for each $n=1,\ldots,N$, there
  exists a unique solution to the system of maximization problems in
  \eqnref{substaticBertrandNashEqDefMaxCondition}.
\end{assump}
Sufficient conditions for the existence of a unique best-response
function for each player are existence of a unique solution to the
first-order conditions:
\begin{equation}\label{eq:FOC}
\pderiv{D_i}{p_i}(p^{\star,n})\left(p^{\star,n}_i - s_i\right) +
D_i(p^{\star,n}) = 0, \qquad i=1,\ldots,n,
\end{equation}
and strict concavity of $\Pi^n_i$ as a function of $p_i$. It is
straightforward to show that the latter is implied if we adopt the
assumption that
$D_i^n(p_1,\ldots,p_n)$ is concave as a function of $p_i$; some weaker
conditions on the $D_i^n$ are discussed in \citet[Chapter 6]{vives99},
but we do not pursue those here. Finally, for a unique intersection of
the best-reponse functions, hence a unique solution to
\eqnref{substaticBertrandNashEqDefMaxCondition}, a well-known
sufficient condition is diagonal dominance of the Hessian of $\Pi_i^n$:
\begin{equation}
\frac{\partial^2\Pi_i^n}{\partial p_i^2} + \sum_{j\neq i} \left\vert
  \frac{\partial^2 \Pi_i^n}{\partial p_i\partial p_j}\right\vert < 0,
\qquad i=1,\ldots,n.
\label{eq:diagDomCond}
\end{equation}
Again, we refer to \citet{vives99} for details.

In the subgames, prices $p_i^{\star,n}$ are non-negative, but the
resulting demands may be negative. Therefore, these are only initial
candidates for the Nash equilibrium of our problem, but they are used
in the proof of the next section. We will also provide an example of
such a Nash equilibrium under linear demand functions.

\renewcommand{\theenumi}{$\langle$\Roman{enumi}$\rangle$}
\subsubsection{Existence and Construction of Nash Equilibrium}

Let $p^{\star}$ denote the vector of prices in equilibrium. We will
see that the Nash Equilibrium will be one of three types:
\begin{enumerate}
\item\label{NEtype1} All $N$ firms price above cost. In this case,
  $p_i^{\star} > s_i$ for all $i=1,\ldots,N$, and the Nash equilibrium
  is simply the $N$-player interior Nash equilibrium given by
  $p^{\star}=\left(p_1^{\star,N},\ldots,p_N^{\star,N}\right)$, where
  the $p_i^{\star,N}$ solve
  \eqnref{substaticBertrandNashEqDefMaxCondition} with $n=N$.

\item\label{NEtype2} For some $0\leq n< N$, firms $1,\ldots,n$ price
  strictly above cost and the remaining firms set price equal to cost.
  In other words, $p_i^{\star} > s_i$ for $i=1,\ldots,n$, and $p_j =
  s_j$ for $j=n+1,\ldots,N$. The first $n$ firms play the interior
  $n$-player sub-game equilibrium as if firms $n+1,\ldots,N$ do not
  exist. These firms are completely ignorable because their costs are
  too high. The Nash equilibrium is
  $p^{\star}=\left(p_1^{\star,n},\ldots,p_n^{\star,n},s_{n+1},\ldots,s_{N}\right)$.

\item\label{NEtype3} For some $(k,n)$ such that $0 \leq k < n\leq N$,
  firms $1,\ldots, k$ price strictly above cost (if $k=0$ then no
  firms price strictly above cost), and the remaining firms set price
  equal to cost. In other words, $p_i^{\star} > s_i$ for
  $i=1,\ldots,k$ and $p_j^{\star} = s_j$ for $j=k+1,\ldots,N$. This
  type differs from Type \ref{NEtype2} in that firms $k+1,\ldots,n$
  are not ignorable: their presence is felt in the pricing decisions
  of firms $1,\ldots,k$, and we say that firms $k+1,\ldots,n$ are on
  the boundary. On the other hand, firms $n+1,\ldots,N$ are completely
  ignorable. This case arises when firms $k+1,\ldots,n$ would want to
  price above cost if they were ignored, but they do not want to price
  above cost in the full sub-game that includes them as a player.
\end{enumerate}
\renewcommand{\theenumi}{(\roman{enumi})}

In order to characterize Type \ref{NEtype3} equilibria, for any fixed $n\in\left\{1,\ldots,N\right\}$ and for any
$k=1,\ldots,n$, let $p^{b,n,n-k}$,  be the vector where for
every $i=1,\ldots,k$ we have
\begin{equation}
p_i^{b,n,n-k} = \argmax_{p\geq s_i}
D^n_i\left(p_1^{b,n,n-k},\ldots,p_{i-1}^{b,n,n-k},p,p_{i+1}^{b,n,n-k},\ldots,p_{k}^{b,n,n-k},s_{k+1},
\ldots,s_{n}\right)\cdot\left(p-s_i\right),  \label{eqn:boundaryopt}
\end{equation}
and for which $p_j^{b,n,n-k} = s_j$ for $j=k+1,\ldots,N$. This means
that firms $1,\ldots,k$ are setting prices by maximizing profit in the
sense of Nash equilibrium given that firms $k+1,\ldots,n$ are on the
boundary and firms $n+1,\ldots,N$ are ignorable. We note that this
solution is different to the solution $p^{\star,k}$ because the demand
function used in their profit maximization is $D_i^n$ and not $D^k_i$.
This is exactly what we mean by the fact that firms $k+1,\ldots,n$ are
on the boundary and hence not ignored. Explicitly, the superscript
$(b,n,n-k)$ stands for boundary, $n$ firms entering into the demand
function, and $n-k$ firms on the boundary, i.e. not ignorable.

The following lemma shows that if firm $n$ would see non-positive
demand at cost, for some fixed set of prices $p_1,\ldots,p_{n-1}$,
then, at the same fixed prices, and with firm $n$ pricing at cost,
firm $(n+1)$ will also see non-positive demand at cost.

\begin{lemma}\label{lemma:negDemandAtCost}Fix an
  $n\in\left\{1,\ldots,N-1\right\}$ and fix $p_1,\ldots,p_{n-1}$.
  Suppose $D^n_n(p_1,\ldots,p_{n-1},s_n) \leq 0$. Then
  $D^{n+1}_{n+1}\left(p_1,\ldots,p_{n-1},s_n,s_{n+1}\right) \leq 0$.
\end{lemma}
\begin{proof}
Recall that $\hat{p}_{n+1}$ is the unique price as a function of
$(p_1,\ldots,p_n)$ that equates the demand of firm $(n+1)$ to zero.
For the fixed $p_1,\ldots,p_{n-1}$ and $s_n$, we have
\begin{equation}
D^{n+1}_{n+1}\left(p_1,\ldots,p_{n-1},s_n,\hat{p}_{n+1}\right) =0.
\end{equation}
Then,
\begin{eqnarray*}
D^n_n\left(p_1,\ldots,p_{n-1},s_n\right) &=& D^{n+1}_n\left(p_1,\ldots,p_{n-1},s_n,\hat{p}_{n+1}\right)\\
&=& D^{n+1}_{n+1}\left(p_1,\ldots,p_{n-1},\hat{p}_{n+1},s_n\right) \\
&\geq& D^{n+1}_{n+1}\left(p_1,\ldots,p_{n-1},\hat{p}_{n+1},s_{n+1}\right) \\
&\geq& D^{n+1}_{n+1}\left(p_1,\ldots,p_{n-1},s_n,s_{n+1}\right),
\end{eqnarray*}
where the last inequality holds if $s_n \leq \hat{p}_{n+1}$.
Alternatively, if $s_n \geq \hat{p}_{n+1}$, we note that then we also
have $s_{n+1}\geq s_n\geq \hat{p}_{n+1}$ and thus
\begin{equation*}
D^{n+1}_{n+1}\left(p_1,\ldots,p_{n-1},s_n,s_{n+1}\right) \leq D^{n+1}_{n+1}\left(p_1,\ldots,p_{n-1},s_n,\hat{p}_{n+1}\right) = 0.
\end{equation*}
Hence, regardless of the relative size of $s_n$ and $\hat{p}_{n+1}$ we
have $D^{n+1}_{n+1}\left(p_1,\ldots,p_{n-1},s_n,s_{n+1}\right) \leq
0$.
\end{proof}

\begin{theorem}There exists a unique Nash equilibrium to the Bertrand game.
\label{thm:generalNE}
\end{theorem}
\begin{proof}
We begin with the lowest cost firm. His equilibrium candidate price is
given by $p_1^{\star,1}$. If $p_1^{\star,1}$ is less than or equal to
$s_1$ then the optimal response of firm 1 is to set price equal to
cost.
By Lemma \ref{lemma:negDemandAtCost}, every other firm has
negative demand at cost. Hence, it is the best response of all firms
to set price at cost. In this case, costs are so high that no firms receive demand in
equilibrium, and we have
\begin{equation}
p^{\star} = \left(s_1,\ldots,s_N\right),
\end{equation}
which is of Type \ref{NEtype2} with $n=0$.
Alternatively, if $p_1^{\star,1} > s_1$, then additional firms may also want to price above cost.

Suppose that for some $n\geq 1$ we have $p_i^{\star,n} > s_i$ for all
$i=1,\ldots,n$. Consider the pricing decision of firm $n+1$. We find
that if
\begin{equation}
D_{n+1}^{n+1}\left(p_1^{\star,n},\ldots,p_n^{\star,n},s_{n+1}\right) \leq 0,
\end{equation}
then firm $(n+1)$ will not want to price above cost because even at
cost they do not receive demand. Furthermore, by Lemma
\ref{lemma:negDemandAtCost}, firms $n+2,\ldots, N$ will also not
receive demand and their best response will thus be to set price at
cost. Hence, we have a Type \ref{NEtype2} equilibrium given by
\begin{equation*}
p^{\star} = \left(p_1^{\star,n},\ldots,p_n^{\star,n},s_{n+1},\ldots,s_N\right).
\end{equation*}
However, if
\begin{equation}\label{firmNp1EntryCheck}
D_{n+1}^{n+1}\left(p_1^{\star,n},\ldots,p_n^{\star,n},s_{n+1}\right) > 0,
\end{equation}
then firm $n+1$ may want to price above cost. We must then distinguish
two cases. The first is where $p_{n+1}^{\star,n+1} > s_{n+1}$. Then
firm $(n+1)$ will want to price according to the interior candidate
price, and all firms with lower cost will also price
at their $(n+1)$-firm interior candidate prices. At this point, we have to consider
the entry decision of the next firm, thereby moving back to the
beginning of this inductive step if $n+1<N$. However, if $n+1 = N$
then we stop and we have a Type \ref{NEtype1} equilibrium given by
\begin{equation*}
p^{\star} = \left(p_1^{\star,N},\ldots,p_N^{\star,N}\right).
\end{equation*}

The second case is where $p_{n+1}^{\star,n+1} \leq s_{n+1}$. Here, by
\eqnref{firmNp1EntryCheck}, firm $(n+1)$ wants to price above cost
when the first $n$ firms are pricing at their interior candidate prices
in the $n$-firm game. But, its cost is too high to receive any demand
at its $(n+1)$-firm candidate price. We say that this firm is on the
boundary. Therefore, firm $(n+1)$ must set price equal to $s_{n+1}$,
because if they were to price strictly above cost it would have to be
an interior candidate price, and we already have seen that this is
not possible for the given $s_{n+1}$. We have thus ruled out both Type
\ref{NEtype1} and Type \ref{NEtype2} equilibria, and the equilibrium of
this game is of Type \ref{NEtype3}.

Hence, the remaining firms solve for an equilibrium with the
$(n+1)$-firm
demand functions, but with $p_{n+1}$ fixed at $s_{n+1}$. This
will result in the prices $p_i^{b,n+1,1}$ for firms $i=1,\ldots,n$. If
$p_n^{b,n+1,1} \geq s_n$, then we can stop and we have
\begin{equation}
p^{\star} = \left(p_1^{b,n+1,1},\ldots,p_n^{b,n+1,1},s_{n+1},\ldots,s_N\right),
\end{equation}
again where we know that all firms with cost greater than firm $(n+1)$
price at cost by Lemma \ref{lemma:negDemandAtCost}.
However, suppose to the contrary that $p_n^{b,n+1,1} < s_n$. Then
$s_n$ is too high to sustain a boundary solution with player $n$ pricing above cost, and we must consider the situation where there is more than one firm on
the boundary. We find $k\in\left\{0,\ldots,n\right\}$ such that
\begin{equation}
p_{k+1}^{b,n+1,n-k} < s_{k+1} \hspace{0.5cm}\textrm{ and
}\hspace{0.5cm}p_k^{b,n+1,n-k+1} \geq s_k. \label{eq:kcond}
\end{equation}
Here $k$ represents the number of firms setting price according to the
boundary optimization \eqnref{eqn:boundaryopt}, and $n-k+1$ is the number of firms on
the boundary.
From \eqnref{eq:kcond}, the best response of each of the
$n-k+1$ boundary firms is to price at cost and, by Lemma
\ref{lemma:negDemandAtCost}, for the remaining higher cost firms to
also price at cost. Meanwhile firms $1,\ldots,k$ choose
prices $p_i^{b,n+1,n-k+1}$ which are greater than their costs.
Thus, we have
\begin{equation}
p^{\star} = \left(p_1^{b,n+1,n-k+1},\ldots,p_k^{b,n+1,n-k+1},s_{k+1},\ldots,s_n,s_{n+1},\ldots,s_N\right).
\end{equation}
\end{proof}
\subsubsection{Nash Equilibrium with Linear Demand}
We give explicit expressions for the Nash equilibrium to the Bertrand
game under the linear demand functions discussed in Section \ref{sec:linDemandDef}.

\begin{prop}\label{prop:linearNE}
There exists a unique equilibrium to the Bertrand game with linear
demand. The type \ref{NEtype1} and \ref{NEtype2} candidate solutions
are given by
\begin{equation}\label{eq:propCandidates}
p_i^{\star,n} = \frac{1}{\left(2b_n+c_n\right)}\left[a_n + c_n \frac{na_n + b_n\sum_{m=1}^ns_m}{\left(2b_n - (n-1)c_n\right)} + b_ns_i\right].
\end{equation}
The type \ref{NEtype3} candidate solutions are given by
\begin{eqnarray}
p_i^{b,n+1,n+1-k} &=& \frac{1}{\left(2b_{n+1}+c_{n+1}\right)}\left[\left(a_{n+1}+c_{n+1}\sum_{m=k+1}^{n+1}s_{m}\right) \nonumber\right. \\
& & \left.
  +c_{n+1}\left(\frac{n\left(a_{n+1}+c_{n+1}\sum_{m=k+1}^{n+1}s_{m}\right)
      + b_{n+1}\sum_{m=1}^k s_m}{2b_{n+1} - (k-1)c_{n+1}}\right) +~
  b_{n+1}s_i\right].\label{eq:propBoundaryEqPricesK}
\end{eqnarray}
The Nash equilibrium is constructed as follows:
\begin{itemize}
\item If $s_1 > \frac{a_1}{b_1}$, then $p^{\star} = \left(s_1,\ldots,s_N\right)$.
\item Else, find $n$ such that $p_i^{\star,n} > s_i, \forall i=1,\ldots,n,$ and $p_{n+1}^{\star,n+1}\leq s_{n+1}$.
    \begin{itemize}
    \item If $s_{n+1} \geq b_{n+1}^{-1}\left(a_{n+1} +
          c_{n+1}\sum_{i=1}^np_i^{\star,n}\right)$, then $p^{\star} = \left(p_1^{\star,n},\ldots,p_n^{\star,n},s_{n+1},\ldots,s_N\right)$,
    \item Else,
\begin{itemize}
\item[$\circ$] if $p_n^{b,n+1,1} > s_n$, then $p^{\star} =
  \left(p_1^{b,n+1,1},\ldots,p_n^{b,n+1,1},s_{n+1},\ldots,s_N\right)$,
    \item[$\circ$] else, find $k<n$ such that
\[ p_i^{b,n+1,n+1-k} > s_i\,\mbox{ for all } i=1,\ldots,k, \quad \mbox{and}
\quad p_{k+1}^{b,n+1,n+1-(k+1)}<s_{k+1}. \]
            Then $p^{\star} =
                        \left(p_1^{b,n+1,n+1-k},\ldots,p_k^{b,n+1,n+1-k},s_{k+1},\ldots,s_N\right)$.
    \end{itemize}
        \end{itemize}
\end{itemize}
\end{prop}

\begin{proof}
  We first show that the first-order condition equation,
  \eqnref{eq:FOC}, has a unique solution. The second-order conditions
  that these are maxima for each player are satisfied as a straightforward
  consequence of $b_n >0$. In order to find a formula
  for $p_i^{\star,n}$, we first solve the unconstrained individual
  firm profit maximization problem to get the best-response function
  for each firm. This results in
\begin{equation}\label{bestRespi}
p_i^{\star,n} = \frac{1}{2}\left(\frac{a_n}{b_n} + \frac{c_n}{b_n}\sum_{j\neq i}p_j^{\star,n} + s_i\right).
\end{equation}
In order to find the intersection of all these functions, we sum
\eqnref{bestRespi} over $i$ to obtain
\begin{equation}\label{sumOfPiStar}
\bar{p}^{\star,n} = \frac{na_n + b_n\bar{s}_n}{\left(2b_n - (n-1)c_n\right)},
\end{equation}
where $\bar{s}_n = \sum_{j=1}^n s_j$, the sum of the first $n$
firms' costs, and $\bar{p}^{\star,n} = \sum_{j=1}^n p_i^{\star,n}$.
Rewriting \eqnref{bestRespi} in terms of $\bar{p}^{\star,n}$ gives
\begin{equation}\label{theCandidatePrice}
p_i^{\star,n} = \frac{1}{\left(2b_n+c_n\right)}\left[a_n + c_n\bar{p}^{\star,n} + b_ns_i\right].
\end{equation}
Thus, our candidate solution $p_i^{\star,n}$, found by solving \eqnref{eq:FOC}, is given by
\eqnref{theCandidatePrice} for $i=1,\ldots,n$ and $p_i^{\star,n} =
s_i$ for $i=n+1,\ldots,N$. This establishes the formulas in
\eqnref{eq:propCandidates}. The $p_i^{\star,n}$ are necessarily
positive because $2b_n > (n-1)c_n$ which follows easily from \eqnref{ABCintermsOfOthers}
and our standing assumption that $B > (N-1)C$ (which is equivalent to
$\beta > \gamma$). Therefore we have a unique Nash equilibrium of the
subgame with positive prices.

The boundary formulas
\eqnref{eq:propBoundaryEqPricesK} are established similarly, and the remainder of the proposition follows
from the proof of Theorem \ref{thm:generalNE}.
\end{proof}

\begin{remark}
The condition for positive prices in the proof, namely $2b_n > (n-1)c_n$, is
exactly the diagonally dominant condition of \eqnref{eq:diagDomCond}.
\end{remark}

\subsection{Discussion of the Static Game}\label{sec:linDuopExample_Equil}
The boundary type of solution we discuss above does not appear to
exist in the literature on Bertrand games, which has primarily
focused on cases where Type \ref{NEtype1} equilibria occur. For
example in the typically-studied case where firms are taken to have
equal costs, the boundary would not exist because all firms would
either price at cost, or they would all play an interior Nash
equilibrium. However, this boundary type of solution may occur when
firms have asymmetric costs. Consider a firm who prices strictly
above cost in a boundary equilibrium. One can think of this firm as
using price to discourage competition from other smaller or less
efficient firms. In the simplest case of two players, a potential
monopolist sets a price below the optimal monopoly price in order to
discourage the entry of a possible competitor. Such practices are
usually termed predatory pricing.

For further discussion, we illustrate with the linear duopoly. We assume we have linear
demand functions in the sense of Section \ref{sec:linDemandDef} with $N=2$ for
fixed constants $A,B$ and $C$, with $B>C$. We use the result of Proposition
\ref{prop:recursionSoln} to re-parameterize the problem in terms of
constants $\alpha,\beta$ and $\gamma$. For a fixed cost $s$, the
optimal price and realized demand in the monopoly problem are
given by:
\begin{equation}\label{eq:monopPriceAndDemand}
p_M^{\star}(s) = \frac{1}{2}\left(\alpha + s\right), \qquad
D_M^{\star}(s) = \frac{1}{2\beta}\left(\alpha - s\right).
\end{equation}
In terms of the players' costs $(s_1,s_2)$,
the interior equilibrium duopoly prices and demands are given by
\begin{eqnarray}
p_i^{\star,2}(s_1,s_2) &=& \alpha\left(\frac{\beta-\gamma}{2\beta -
    \gamma}\right) + \frac{\beta}{(4\beta^2-\gamma^2)}\left(2\beta s_i +
  \gamma s_j\right), \qquad i=1,2; j\neq i, \label{eqn:pstar2}\\
D_i^{\star}(s_1,s_2) &=& \frac{\alpha}{\beta + \gamma} -
\frac{\beta}{(\beta^2 - \gamma^2)}p_i^{\star}(s_1,s_2) +
\frac{\gamma}{(\beta^2 - \gamma^2)}p_j^{\star}(s_1,s_2). \label{eqn:Dstar2}
\end{eqnarray}
Finally, if the boundary case arises, the equilibrium prices are given by
\begin{equation}\label{eq:duopBoundaryPrice}
p^{b,2,1}_i(s_1,s_2) =
\frac{1}{2}\left(\frac{\alpha(\beta-\gamma)+\gamma s_j}{\beta} +
  s_i\right), \qquad\mbox{and}\qquad p_j^{\star} = s_j,
\end{equation}
which can be found from the two-player profit maximization problem
under the assumption that one's opponent sets price equal to cost. If
$i$ denotes the lower cost firm such that $0\leq s_i\leq s_j$, then
the Nash equilibrium price strategies are given by
\begin{equation}\label{eq:staticSoln1}
p_i^{\star} = \max\left(s_i, \left\{\begin{array}{ccc} p_M^{\star}(s_i) & \textrm{ if } & D_j(p_M^{\star}(s_i),s_j) \leq 0 \\
p_i^{b,2,1}(s_1,s_2) & \textrm{ if } & D_j(p_M^{\star}(s_i),s_j)  > 0 \textrm{ and } p_j^{\star,2} \leq s_j \\
p_i^{\star,2} & \textrm{ else} & \end{array}\right.\right),
\end{equation}
and $p_j^{\star} = \max\left(s_j,p_j^{\star,2}\right)$. In the case
$s_1 = s_2 = s$, this simplifies to $p_1^{\star} = p_2^{\star} =
\max\left(s,\frac{\alpha(\beta-\gamma)+\beta
    s}{2\beta-\gamma}\right)$.

We examine the above solutions in more detail. Let us first note that
if $p_1^{\star,2} < s_1$, then a duopoly is not sustainable. This
occurs if and only if
\begin{equation}
\phi^1(s_1) \triangleq \left(\frac{2\beta^2-\gamma^2}{\beta\gamma}\right)s_1 - \frac{\alpha}{\beta\gamma}\left(\beta-\gamma\right)\left(2\beta+\gamma\right) > s_2.
\end{equation}
Similarly, we note that if
$D_1^{\star}\left(s_1,p_M^{\star}(s_2)\right) < 0$, then Firm 2 has a
monopoly. This occurs if and only if
\begin{equation}
\phi^2(s_1) \triangleq \frac{2\beta}{\gamma}s_1 -
\frac{\alpha}{\gamma}\left(2\beta - \gamma\right) > s_2.
\end{equation}
Therefore, for $s_2 \in \left(\phi^2(s_1), \phi^1(s_1)\right)$, Firm
2 cannot sustain a monopoly, but neither is a duopoly sustainable.
This is the situation where Firm 1 is on the boundary. By the
symmetry of the game, we can also use $\phi^1$ and $\phi^2$ to
characterize where Firm 2 is on the boundary, and where Firm 1 has a
monopoly. We can fully characterize the type of game in the space of
$(s_1,s_2)$ through these two functions. First, note $\phi^1(\alpha)
= \phi^2(\alpha) = \alpha,$ and $\phi^1(s) - \phi^2(s) =
\frac{\gamma}{\beta}\left(\alpha - s\right)$.  We see in Figure
\ref{duopSolnCharacterization} that as long as $|s_1 - s_2|$ is
relatively small, a duopoly will be sustainable. We also note that
the dependence of $\phi^1(s) - \phi^2(s)$ on $\gamma$ implies that
as $\gamma$ decreases, the size of the boundary area will also
decrease.
\begin{figure}[htb]
\begin{center}
\setlength{\unitlength}{0.9cm}
\begin{picture}(9,9)(0,0)
\put(0,0){\vector(0,1){7}}
\put(-0.5,6.5){$s_2$}
\put(0,0){\vector(1,0){7}}
\put(6.5,-0.5){$s_1$}
\multiput(6,-0.2)(0,0.5){13}{\line(0,1){0.3}}
\put(5.9,-0.5){$\alpha$}
\multiput(-0.2,6)(0.5,0){13}{\line(1,0){0.3}}
\put(-0.6,5.9){$\alpha$}

\put(3,0){\line(1,2){4}}
\multiput(4.5,0)(0.2,0.8){10}{\line(1,4){0.1}}

\put(0,3){\line(2,1){8}}
\multiput(0,4.5)(.8,0.2){10}{\line(4,1){0.5}}

\put(5.2,1.5){$M_2$}
\put(4,1.5){$B_2$}
\put(1,2){$Duopoly$}
\put(1,4){$B_1$}
\put(0.9,5.3){$M_1$}

\put(7,7.5){$\phi^1(s_1)$}
\put(5,7.3){$\phi^2(s_1)$}

\put(8,6.8){$\phi^1(s_2)$}
\put(7.2,5.8){$\phi^2(s_2)$}

\put(6.5,4){$M_i = $ Firm $i$ Monopoly}
\put(6.5,3){$B_i = $ Firm $j$ on Boundary}
\end{picture}
\end{center}
\caption{Characterization of solution in cost space}\label{duopSolnCharacterization}
\end{figure}
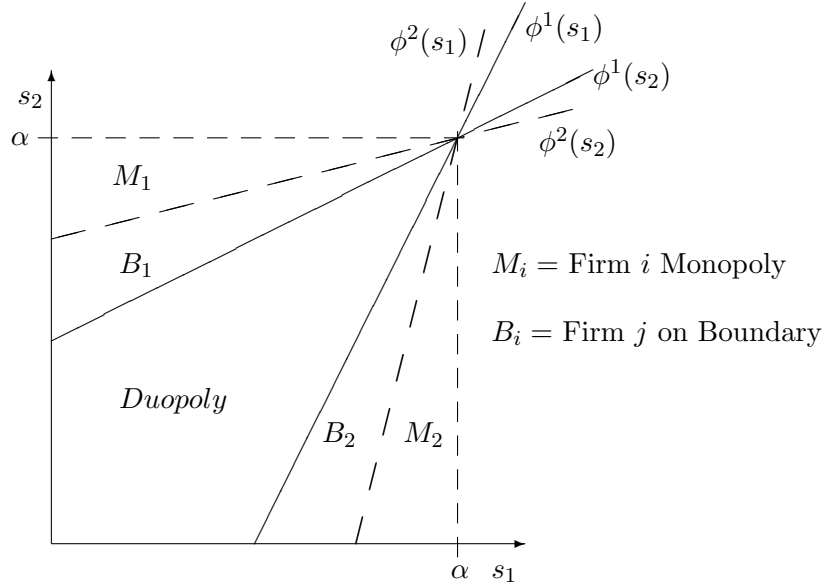
Consider a fixed set of costs $(s_1,s_2)$ such that $s_1>s_2$ and $s_2 = \phi^2(s_1)$. This is the point where the game transitions from the boundary solution to the monopoly solution. Using the monopoly price and boundary price given above, we find at this transition point that
\begin{equation*}
p_M^{\star}\left(\phi^2(s_1)\right) - p_2^{b,2,1}\left(s_1,\phi^2(s_1)\right) = \frac{\gamma}{2\beta}\left(\alpha - s_1\right)>0,
\end{equation*}
and thus there is a jump in the equilibrium price of Firm 2 as the game transitions from Firm 1 on the boundary to Firm 2 being a monopoly.

\section{Differential Game}\label{sec:diffGame}
The single-period game provides only the beginning of an insight into the pricing decisions of firms. In reality, firms make their decisions dynamically through time.
We consider a market in which there are $N$ possible firms, each of which has a fixed lifetime capacity of production at time $t=0$ denoted by $x_i(0)$, and where $x_i(t)$ denotes the remaining capacity at time $t$. The firms in this market produce only a single good and thus, one can use the terminology firm and product interchangably. In this sense, $x_i(t)$ represents the remaining amount of a certain product that will be sold. Therefore, when $x_i=0$, no more of the product will be sold, the firm has exhausted its capacity and is out of business. Thus, $x_i$ is not to be confused with inventory which is typically replenishable. Our point of view abstracts from the microscopic level of inventory fluctuations to the level of lifetime production. For simplicity of notation, we consider the cost of production in the dynamic game to be zero, but we will see there are shadow costs associated with scarcity of goods as they run down.

Each firm $i$ chooses a Markovian dynamic pricing strategy, $p_i = p_i(\mathbf{x}(t))$ where $\mathbf{x}(t) = \left(x_1(t),\ldots,x_N(t)\right)$. This is the price at which consumers can purchase a unit of the good produced by firm $i$. The firms in this market produce substitute, but not perfectly substitutable, goods. As in Section \ref{sec:StaticGame}, given these prices, each firm $i=1,\ldots,N$ expects the market to demand at a rate $D_i(p_1,p_2,\ldots,p_N)$, but actual demands from the market may see short term unpredictable fluctuations.
We model them in the simplest way:
\begin{equation}\label{eq:dynActualDemand}
d_i(t) = D_i\left(p_1,p_2,\ldots,p_N\right) - \sigma_i\dot{\epsilon}_i(t),
\end{equation}
where $\left\{\dot{\epsilon}_i(t)\right\}_{i=1,\ldots,N}$ are correlated Gaussian white noise sequences. Consequently, the dynamics of the lifetime capacity of the firms is given by $dx_i(t) = -d_i(t)~dt$. This leads to the controlled stochastic differential equations
\begin{equation}\label{eq:inventoryDynamics}
dx_i(t) = -D_i\big(p_1(\mathbf{x}(t)),\ldots,p_N(\mathbf{x}(t))\big)
dt + \sigma_i~dW_i(t), \hspace{0.1cm}\textrm{ if } x_i > 0, ~~
i=1,\ldots,N,
\end{equation}
where $\{W_i(t)\}_{i=1,\ldots,N}$ are correlated Brownian motions. If
$x_i(t) = 0$ for any $t$, then $x_i(s) = 0$ for all $s \geq t$: random
shocks cannot resuscitate a firm that has gone out of business.
The use of this type of additive shock in demand is common in the
economics literature, for example \citet{arrow51}. The use of a
Brownian motion for the demand flow can be found in various sources,
for example \citet{bather66}. An alternative model for demand
uncertainty in the literature is to consider the process of customer
arrival as a Poisson process, see for example \citet{zeevi09} for
recent work in this direction.

\subsection{The Linear Demand Duopoly Game}

Now that we have fully specified the dynamics of the firms' remaining
lifetime capacities, we move on to the actual study of the dynamic
game. The analysis can be done for an arbitrary number of players $N$,
but, to simplify the exposition, we focus on the case $N=2$,
i.e. duopoly. Additionally, we focus on the case of a linear demand
system. This will allow us to be more explicit in our actual results. See Appendix \ref{append:sec:NplayerPDE} for a discussion of the
$N$-player linear demand game.

Given initial lifetime capacity $x_i(0) > 0$, player $i=1,2$ seeks to maximize
his expected discounted lifetime profit
\begin{equation}
\expec \left\{\int_0^{\infty}
  e^{-rt}p_i(x(t))D_i\left(p_1\left(x(t)\right),p_2\left(x(t)\right)\right)\indicator_{\left\{x_i(t)
      > 0\right\}}~dt\right\},
\end{equation}
where $r>0$ is a discount rate and $D_i$ are the actual demands
constructed in Definition \ref{actdem} using the linear demand functions
given in Section \ref{sec:linDemandDef}. We restrict attention to
Markov Perfect Nash equilibria in order to rule out equilibria with
undesirable properties such as non-credible threats (see, for example,
\citet[Chapter 13]{fudenberg91}). This means that we are looking for a
pair
$\left(\bar p_1^{\star}(\mathbf{x}(t)), \bar p_2^{\star}(\mathbf{x}(t))\right)$
such that for $i=1,2$, $j\neq i$, and for all
$\mathbf{x}(0)\in\mathbb{R}^2_+ $,
\begin{eqnarray*}
  \expec\left\{ \int_0^{\infty}
    e^{-rt}\bar p^{\star}_i(x(t))D_i\left(\bar
      p_i^{\star}\left(\mathbf{x}(t)\right),\bar
      p^{\star}_j\left(\mathbf{x}(t)\right)\right)\indicator_{\left\{x_i(t)
        > 0\right\}}~dt \right\} \geq \\
\expec\left\{ \int_0^{\infty}
    e^{-rt}p_i(x(t))D_i\left(p_i\left(\mathbf{x}(t)\right), \bar p^{\star}_j
      \left(\mathbf{x}(t)\right)\right)\indicator_{\left\{x_i(t)
        > 0\right\}}~dt\right\},
\end{eqnarray*}
for any Markov strategy $p_i$ of player $i$. (The overbar in $\bar
p_i^{\star}$  is used to distinguish the dynamic Nash equilibrium from
the equilibrium of the static game in Section \ref{sec:StaticGame}).

We define the value functions of the two firms by the coupled
optimization problems
\begin{equation}
  V_i(x_1,x_2) = \sup_{p_i\geq 0} \expec\left\{\int_0^{\infty}
    e^{-rt}p_i(\mathbf{x}(t))D_i\left(p_1\left(\mathbf{x}(t)\right),p_2
      \left(\mathbf{x}(t)\right)\right)\indicator_{\left\{x_1(t)
        > 0\right\}}~dt\right\}, \qquad i=1,2. \label{eq:Vi}
\end{equation}
Then, by a dynamic programming argument for nonzero-sum differential
games (see, for example, \citet[Section 8.2]{friedman06},
\citet[Section 6.5.2]{basar95}, or \citet[Section 4.2]{dockner00}),
these value functions, if they have sufficient regularity, satisfy the
following system of PDEs:
\begin{equation}\label{eq:PDEsystem1}
\mathcal{L}V_i + \sup_{p_i\geq 0}\left\{ -
  D_1\left(p_1,p_2\right)\pderiv{V_i}{x_1} -
  D_2\left(p_1,p_2\right)\pderiv{V_i}{x_2} +
  p_iD_i\left(p_1,p_2\right)\right\} -rV_i = 0
\end{equation}
for $i=1,2$, where
\begin{equation*}
  \mathcal{L} = \frac{1}{2}\sigma_1^2\pderivs{}{x_1} +
  \rho\sigma_1\sigma_2\pcderivs{}{x_1}{x_2} +
  \frac{1}{2}\sigma^2_2\pderivs{}{x_2},
\end{equation*}
and $\rho$ is the correlation coefficient of the Brownian motions:
$\expec\left\{dW_1 dW_2\right\} = \rho\, dt$.

When the parameter $\gamma$ is not too large, both players are close
to being monopolists in disjoint markets for their own goods, and we
can expect that the dynamic Nash equilibrium $\left(\bar
  p_1^{\star}(\mathbf{x}(t)), \bar p_2^{\star}(\mathbf{x}(t))\right)$
is such that both demands $D_i\left(\bar p_1^{\star}(\mathbf{x}(t)),
  \bar p_2^{\star}(\mathbf{x}(t))\right)$ remain strictly positive
while $x_i(t)>0$.  We shall
find that this is indeed the case for small enough $\gamma$ in the
asymptotic solution of Section \ref{asymp} and the numerical solutions in
Section \ref{sec:numericalAnalysis}.

When both demands are positive, we see easily from Propositions
\ref{prop:linDemand} and \ref{prop:recursionSoln} that the linear
demand functions satisfy the relationship
\begin{equation}\label{eq:demandDecomp}
D_j(p_1,p_2) = D_M(p_j) - \frac{\gamma}{\beta}D_i(p_1,p_2), \qquad
j\neq i,
\end{equation}
which allows us to re-write \eqnref{eq:PDEsystem1} as
\begin{equation}\label{eq:PDEsystem1a}
\mathcal{L}V_i - D_M(p_j)\pderiv{V_i}{x_j} + \sup_{p_i\geq 0}\left\{
  D_i\left(p_1,p_2\right)\cdot\left[p_i - \left(\pderiv{V_i}{x_i}
      -\frac{\gamma}{\beta}\pderiv{V_i}{x_j}\right)\right] \right\}
-rV_i = 0, \qquad i=1,2.
\end{equation}
We now observe that the Nash equilibrium problem in the two PDEs in
\eqnref{eq:PDEsystem1a} is exactly a static Nash equilibrium problem
for a two-player Bertand game, but with costs
\begin{equation}\label{eq:theShadowCosts}
S_i\left(\mathbf{x}\right)\triangleq
\pderiv{V_i}{x_i}\left(\mathbf{x}\right) -
\frac{\gamma}{\beta}\pderiv{V_i}{x_j}\left(\mathbf{x}\right), \qquad i=1,2.
\end{equation}

Given the unique Nash equilibrium
$p_i^\star\left(S_1\left(\mathbf{x}\right),S_2\left(\mathbf{x}\right)\right)$
of this static problem from Proposition \ref{prop:linearNE}, the PDE
system is simply
\begin{equation}\label{eq:PDEsystem2}
\mathcal{L}V_i -
D_M\left(p_j^{\star}\left(S_1\left(\mathbf{x}\right),S_2\left(\mathbf{x}\right)\right)\right)\pderiv{V_i}{x_j} +
G_i\left(S_1\left(\mathbf{x}\right),S_2\left(\mathbf{x}\right)\right)
- rV_i = 0,  \hspace{0.5cm} i,j=1,2; j\neq i,
\end{equation}
where we define $G_i(s_1,s_2) =
D_i\left(p_1^{\star},p_2^{\star}\right)\left(p_i^{\star} - s_i\right)$
as the equilibrium profit function of the static game.

The domain of the PDE problem is $x_1>0, x_2>0$. When one firm runs
out of capacity, the other has a monopoly. We denote by $v_M(x)$ the
value function of a monopolist with remaining capacity $x$, which we
will study in the next section. On $x_2=0$, $x_1>0$, Firm 1 has a
monopoly, so $V_2(x_1,0) \equiv 0$ and $V_1(x_1,0) = v_M(x_1)$. On
$x_1=0$, $x_2>0$, Firm 2 has a monopoly, so $V_1(0,x_2)\equiv 0$ and
$V_2(0,x_2) = v_M(x_2)$.

As is well known, it is extremely difficult to provide existence and
regularity results for systems of PDEs arising from nonzero-sum
differential games and we do not attempt to do so here.  Some results
on weak solutions are found in \citet{bf,bensoussan02}, for related
problems on smooth bounded domains with absorbing boundary
conditions. In contrast, zero-sum games, which are characterized by a
scalar equation, have a well studied viscosity theory; see, for
example \citet{flemingSoug89}. Mean Field Games are an intermediate
case characterized by a system of two PDEs and some regularity results
exist (\citet{lasryLions07jap}). We also mention some analytical
progress can be made in nonzero-sum stochastic differential games of
Dynkin type, that is games on stopping times; see \citet{hamadene10}.

In the stochastic game, the intuition is that the elliptic operator
$\mathcal{L}$ will provide regularity which is supported in the
numerical results of Section \ref{sec:numericalAnalysis}. In the
non-stochastic game, when $\gamma$ is small enough, we obtain
regular asymptotic approximations in Section \ref{asymp} because
the strength of competition between firms is weak.

\subsection{Monopoly Problem}
When one firm has a monopoly over the market, the dynamics for the
firm's remaining capacity, $x(t)$, is given by
\begin{equation*}
dx(t) = -D_M\left(p(x(t))\right)dt + \sigma dW(t),
\end{equation*}
where $W$ is a Brownian motion. The value function of the monopoly
firm as a function of its initial capacity $x_0 = x \in\mathbb{R}_+$ is
defined to be the maximum expected discounted lifetime profit
\begin{equation}\label{eq:monopolyValueFunction}
v_M(x) \triangleq \sup_{p\geq0} \expec\left\{ \int_0^{\infty}
  e^{-rt}p\left(x(t)\right)D_M\left(p\left(x(t)\right)\right)\indicator_{\left\{x(t)>0\right\}}~dt
\right\}.
\end{equation}
The associated Bellman equation for this stochastic control problem is the ODE
\begin{equation*}\label{eq:monopolyBellmanEqn}
\frac{1}{2}\sigma^2v_M^{\prime\prime} +\sup_{p\geq 0} \left\{D_M(p)\left(p-v_M^{\prime}\right)\right\}-rv_M=0,
\end{equation*}
with boundary condition $v_M(0) = 0$. We look for solutions in which $\lim_{x\rightarrow\infty}v_M^{\prime}(x) = 0$.

As we are working in the case of linear demands, we find from \eqnref{eq:monopPriceAndDemand}:
\begin{equation}\label{themonopODE}
\frac{1}{2}\sigma^2v_M^{\prime\prime} + \frac{1}{4\beta}\left(v_M^{\prime}-\alpha\right)^2 -rv_M = 0.
\end{equation}
In the case $\sigma = 0$, the monopoly ODE is given by
\begin{equation}\label{linearODEAnalysisZeroSigODE}
\frac{1}{4\beta}\left(v_M^{\prime}-\alpha\right)^2 -rv_M = 0,
\end{equation}
and we can find an explicit solution.
\begin{prop}\label{prop:monopexplicit} The value function for the monopoly with $\sigma=0$ is
\begin{equation}\label{zeroSigFinalSolution}
v_M(x) = \frac{\alpha^2}{4\beta r}\left[\mathbf{W}\left(-e^{-\mu x-1}\right) + 1\right]^2,
\end{equation}
where $\mu = (2\beta r)/\alpha$ and $\mathbf{W}$ is the Lambert $\mathbf{W}$
function defined by the relation $Y = \mathbf{W}(Y) e^{\mathbf{W}(Y)}$
with domain $Y\geq -e^{-1}$.
\end{prop}
\begin{proof}
It is straightforward to check that the Lambert $\mathbf{W}$ function
satisfies $\mathbf{W}(z) < 0$ for $z\in[-e^{-1},0)$,
$\mathbf{W}\left(-e^{-1}\right) = -1$, $\mathbf{W}(0) = 0$, and
$\mathbf{W}^{\prime}(z) = \mathbf{W}(z)/\left(z\left(1+\mathbf{W}(z)\right)\right)$ for $z>-e^{-1}$,
and therefore that \eqnref{zeroSigFinalSolution} indeed satisfies
\eqnref{linearODEAnalysisZeroSigODE} and the boundary condition
$v_M(0) = 0$. We note that the restriction of the domain of
$\mathbf{W}$ to $[-e^{-1},\infty)$ is sufficient as the argument to
$\mathbf{W}$ in \eqnref{zeroSigFinalSolution} is equal to $-e^{-1}$
when $x=0$ and increases to zero as $x$ increases to infinity.
\end{proof}
Inserting $s= v_M^{\prime} = -\alpha\mathbf{W}\left(-e^{-\mu
    x-1}\right)$, into the optimal monopoly price strategy and demand
functions given in \eqnref{eq:monopPriceAndDemand}, we find
$p^{\star}_M\left(v_M^{\prime}(x)\right) =
\frac{\alpha}{2}\left(1-\mathbf{W}\left(-e^{-\mu x-1}\right)\right)$.
We plot this function in Figure \ref{odeMonopPrice}. We note that the
price at the zero capacity level is given by $\alpha$, and in the
limit as $x \to \infty$, $p^{\star}_M \rightarrow \frac{\alpha}{2}$.
The demand is given by
$D_M^{\star}\left(v_M^{\prime}(x)\right) =
\frac{1}{2\beta}\left(\alpha - v_M^{\prime}(x)\right) =
\frac{\alpha}{2\beta}\left(1+\mathbf{W}\left(-e^{-\mu
      x-1}\right)\right)$, which is strictly positive for all $x >
0$.
\begin{figure}[htp]
\begin{center}
  \includegraphics[width=3.7in]{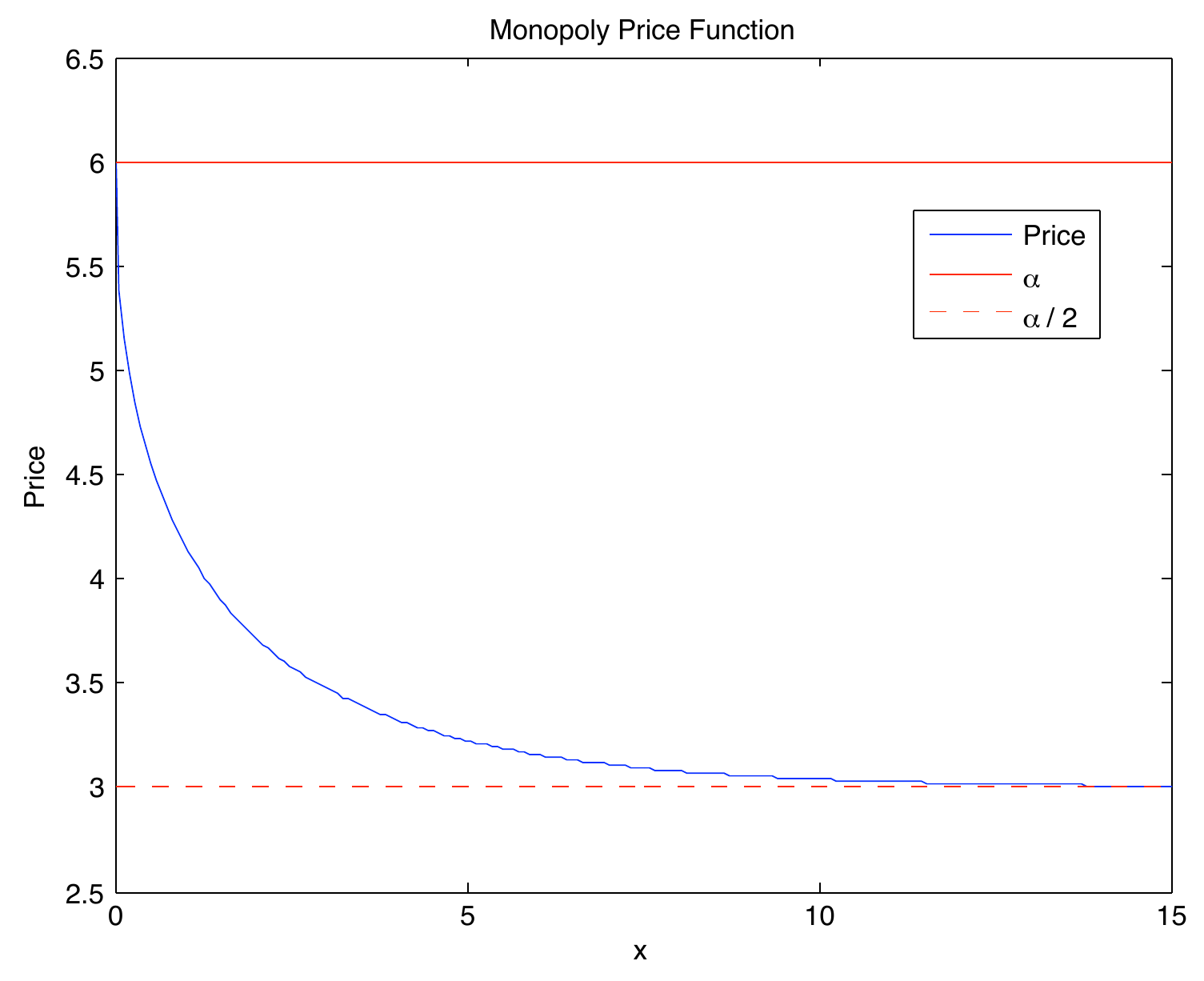}\\
  \caption{$p^{\star}_M(x)$ vs $x$ for $\sigma=0$}\label{odeMonopPrice}
  \end{center}
\end{figure}

\subsection{Small Degree of Substitutability Asymptotics under
  Deterministic Demand\label{asymp}}
We further analyze the non-stochastic (or ordinary) differential game ($\sigma_1 = \sigma_2
= 0$) in which the monopoly problem is explicitly solvable, as in
Proposition \ref{prop:monopexplicit}. We first note that $\gamma=0$ is equivalent to stating that firms have
independent goods in the sense that they operate in markets without
competing with one another. Hence, they have monopolies in their own
markets, and it is clear that $V_i(x_1,x_2) = v_M(x_i)$, where $v_M$
is the monopoly value function given in \eqnref{zeroSigFinalSolution}. When $\gamma > 0$, firms produce goods that are
substitutable and are actually in competition with one another. We
construct a perturbation expansion around the non-competitive case for
small $\gamma>0$ to view the effects of a small amount of competition.

Recall our PDE system \eqnref{eq:PDEsystem2}, with $\mathcal{L}
\equiv 0$.
We look for an approximation to the solution of this system of PDEs of the form
\begin{equation}\label{eqn:expansionAnsatz}
V_i(x_1,x_2) = v_M(x_i) + \gamma v_i^{(1)}(x_1,x_2) + \gamma^2v_i^{(2)}(x_1,x_2) + \cdots .
\end{equation}
We use the static equilibrium price and demand functions that are given in
\eqnref{eqn:pstar2}-\eqnref{eqn:Dstar2}, because we are working under the standing assumption that we have a Type \ref{NEtype1} equilibrium.
We first expand $S_i\left(\mathbf{x}\right) =
s_i^{(0)}\left(\mathbf{x}\right) + \gamma
s_i^{(1)}\left(\mathbf{x}\right) + \gamma^2
s_i^{(2)}\left(\mathbf{x}\right)+\cdots$, where from
\eqnref{eq:theShadowCosts}, we have
\begin{equation}
s_i^{(0)} = v_M^{\prime}(x_i), ~~ s_i^{(1)} = \pderiv{v_i^{(1)}}{x_i}, ~~ s_i^{(2)} = \pderiv{v_i^{(2)}}{x_i} - \frac{1}{\beta}\pderiv{v_i^{(1)}}{x_j}.
\end{equation}
Making use of \eqnref{eq:monopPriceAndDemand}, the relevant expansions in powers of $\gamma$ are
\begin{eqnarray}
p_i^{\star}(s_1,s_2) &=& p_M^{\star}(s_i^{(0)}) -
\frac{\gamma}{2}D_M^{\star}(s_j^{(0)}+2\beta s_i^{(1)}) -
\frac{\gamma^2}{4\beta}D_M^{\star}(s_i^{(0)}+2\beta s_j^{(1)} +
  4\beta^2 s_i^{(2)}) + \cdots,\label{eqn:exp1}\\
D_i^{\star}\left(s_1,s_2\right) &=& D_M^{\star}(s_i^{(0)})
- \frac{\gamma}{2\beta}D_M^{\star}(s_j^{(0)} - 2\beta
  s_i^{(1)}) -
\frac{3\gamma^2}{4\beta^2}D_M^{\star}(s_i^{(0)}+\frac{2}{3}\beta
  s_j^{(1)} - \frac{4}{3}\beta^2 s_i^{(2)}) + \cdots,\label{eqn:exp2}\\
D_M\left(p_i^{\star}\right) &=& D_M^{\star}(s_i^{(0)}) +
\frac{\gamma}{2\beta}D_M^{\star}(s_j^{(0)} + 2\beta
  s_j^{(1)}) +
\frac{\gamma^2}{4\beta^2}D_M^{\star}(s_i^{(0)} + 2\beta s_j^{(1)}
  + 4\beta^2 s_i^{(2)}) + \cdots. \label{eqn:exp3}
\end{eqnarray}

We define
\begin{equation}\label{eqn:defOflilq}
q(x) \triangleq D_M^{\star}\left(v_M^{\prime}(x)\right) = \frac{1}{2\beta}\left(\alpha - v_M^{\prime}\left(x\right)\right).
\end{equation}
Then, inserting \eqnref{eqn:expansionAnsatz} into
\eqnref{eq:PDEsystem2}, using \eqnref{eqn:exp1}-\eqnref{eqn:exp3}, and
comparing terms in $\gamma$ and $\gamma^2$, give that $v_i^{(1)}$ and $v_i^{(2)}$ satisfy
\begin{eqnarray}
q(x_1)\pderiv{v^{(1)}_i}{x_1} + q(x_2)\pderiv{v^{(1)}_i}{x_2} + rv^{(1)}_i &=& -q(x_1)q(x_2), \label{eqn:expansionPDE2} \\
q(x_1)\pderiv{v_i^{(2)}}{x_1} + q(x_2)\pderiv{v_i^{(2)}}{x_2} + rv_i^{(2)} &=& \frac{1}{2\beta}\left(\pderiv{v_j^{(1)}}{x_j} + q(x_i)\right)\cdot\left(\pderiv{v_i^{(1)}}{x_j} + q(x_i)\right) \nonumber \\
&~&+ \frac{1}{4\beta}\left(\pderiv{v_i^{(1)}}{x_i} + q(x_j)\right)^2 - \frac{3}{2\beta}\left(q(x_i)\right)^2,\label{eqn:expansionPDE3}
\end{eqnarray}
for $i=1,2$ and $j\neq i$,  with boundary conditions $v^{(1)}_1(x_1,0) = v^{(2)}_1(x_1,0)  =0$ and $v^{(1)}_1(0,x_2) =v^{(2)}_1(0,x_2)= 0$.
\begin{prop}\label{prop:solnOfPDE1}
The solution $v_i^{(1)}$ is given, for $x_1 > x_2$, by
\begin{equation}\label{eq:v11Solution1}
v^{(1)}_1(x_1,x_2) =
\frac{\alpha^2}{4\beta^2r}\left(e^{-rQ(x_2)}\left(1+rQ(x_2)\right) -
  e^{-rQ(x_1)}\left(1-rQ(x_2)\right) +
  e^{-r\left(Q(x_1)+Q(x_2)\right)} - 1\right),
\end{equation}
where
\begin{equation}\label{expressionForQ}
Q(x) \triangleq \int_0^x \frac{1}{q(u)} du = -\frac{1}{r}\log\left(-\mathbf{W}\left(-e^{-\mu x-1}\right)\right),
\end{equation}
and, for $x_2 \geq x_1$, by reversing the roles of $x_1$ and $x_2$ in \eqnref{eq:v11Solution1}.
The solution for $v^{(1)}_2$ is clearly the same, i.e. $v_2^{(1)} \equiv v_1^{(1)}$.
\end{prop}
\begin{proof}
  The first step is to make the change of variables $(\xi,\eta) =
  \left(Q(x_1),Q(x_2)\right)$ and $u(\xi,\eta) =
  e^{\frac{r}{2}\left(\xi+\eta\right)}v_i^{(1)}\left(Q^{-1}(\xi),Q^{-1}(\eta)\right)$
  in \eqnref{eqn:expansionPDE2}, which gives
\begin{equation}\label{expansionPDEforU}
\pderiv{u}{\xi} + \pderiv{u}{\eta} = f(\xi,\eta),\qquad \xi,\eta>0;
\qquad u(\xi,0)=u(0,\eta)=0,
\end{equation}
where $f(\xi,\eta) \triangleq
-e^{\frac{r}{2}\left(\xi+\eta\right)}q\left(Q^{-1}(\xi)\right)q\left(Q^{-1}(\eta)\right).$
We see by the symmetry of this equation that $u(\xi,\eta)
= u(\eta,\xi)$. We first suppose that $\xi > \eta$ and solve the PDE
with the boundary condition $u(\xi,0) = 0$. The other half of the
solution can be obtained by symmetry.
The solution is
\begin{equation}
u(\xi,\eta) = \int_0^{\eta} f(s+\xi-\eta,s)~ds
= -\int_0^{\eta} e^{\frac{r}{2}(\xi - \eta +
  2s)}q\left(Q^{-1}\left(s+\xi-\eta\right)\right)q\left(Q^{-1}\left(s\right)\right)~ds. \label{solutionForUIntegralForm}
\end{equation}
By the definition of $q(x)$ in
\eqnref{eqn:defOflilq} and $v_M$ in \eqnref{zeroSigFinalSolution}, we have
$q(x) = \frac{\alpha}{2\beta}\left[1 + \mathbf{W}\left(-e^{-\mu
      x-1}\right)\right]$. This leads to \eqnref{expressionForQ} since
the range of $\mathbf{W}\left(-e^{-\mu x-1}\right)$ is $(-1,0)$. From
properties of the Lambert $\mathbf{W}$ function, it follows
easily that $-e^{-rs} = \mathbf{W}\left(-e^{-\mu Q^{-1}(s)-1}\right)$,
and hence
\begin{equation}
q\left(Q^{-1}(s)\right) = \frac{\alpha}{2
  \beta}\left(1+\mathbf{W}\left(-e^{-\mu Q^{-1}(s)-1}\right)\right) =
\frac{\alpha}{2\beta}\left(1 - e^{-rs}\right).
\end{equation}
We can now easily compute the integral in
\eqnref{solutionForUIntegralForm}. After restoring the
transformations, we obtain \eqnref{eq:v11Solution1}.
\end{proof}

\begin{remark}
It can be verified by direct computation that the solutions
$v_i^{(1)}$ in Proposition \ref{prop:solnOfPDE1}
are $\mathcal{C}^1$ on the line $x_1 =
x_2$.
One can also solve the PDEs in \eqnref{eqn:expansionPDE3} to obtain a
second-order correction for the value functions. We present this
solution in Appendix \ref{appendix:2ndOrderPDE}.
\end{remark}

\begin{remark}
  While we do not give a formal convergence proof for the asymptotic
  approximation, we note that, since $v_i^{(1)}$ and $v_i^{(2)}$ are
  bounded with bounded continuous first-derivatives, the error terms
  $E_i$ defined by $V_i=v_M+\gamma v_i^{(1)} + \gamma^2 v_i^{(2)} +
  E_i$ solve
\[ q(x_1)\pderiv{E_i}{x_1} + q(x_2)\pderiv{E_i}{x_2} + rE_i =
O(\gamma^3) \] with $E_i(x_1,0) = E_i(0,x_2)=0$. It follows from here
that, for fixed $x_1,x_2>0$, $|E_i(x_1,x_2)|=O(\gamma^3)$ as $\gamma\downarrow0$.
\end{remark}

\subsection{Discussion of Asymptotic Solution}
We plot $v_1^{(1)}$ in Figure \ref{fig:v1_and_v2}\subref{fig:v11_surface}. Intuitively, it decreases from zero on the
axes, because the first order correction decreases the value of the
game due to the transition from a one-player monopoly to a two-player
duopoly game. Hence, the greater the value of $\gamma$, the lower is
the lifetime profit of an individual firm.
\begin{figure}[htbp]
   \centering
   \subfigure[$v_1^{(1)}$: Surface]{
   \includegraphics[width=3in]{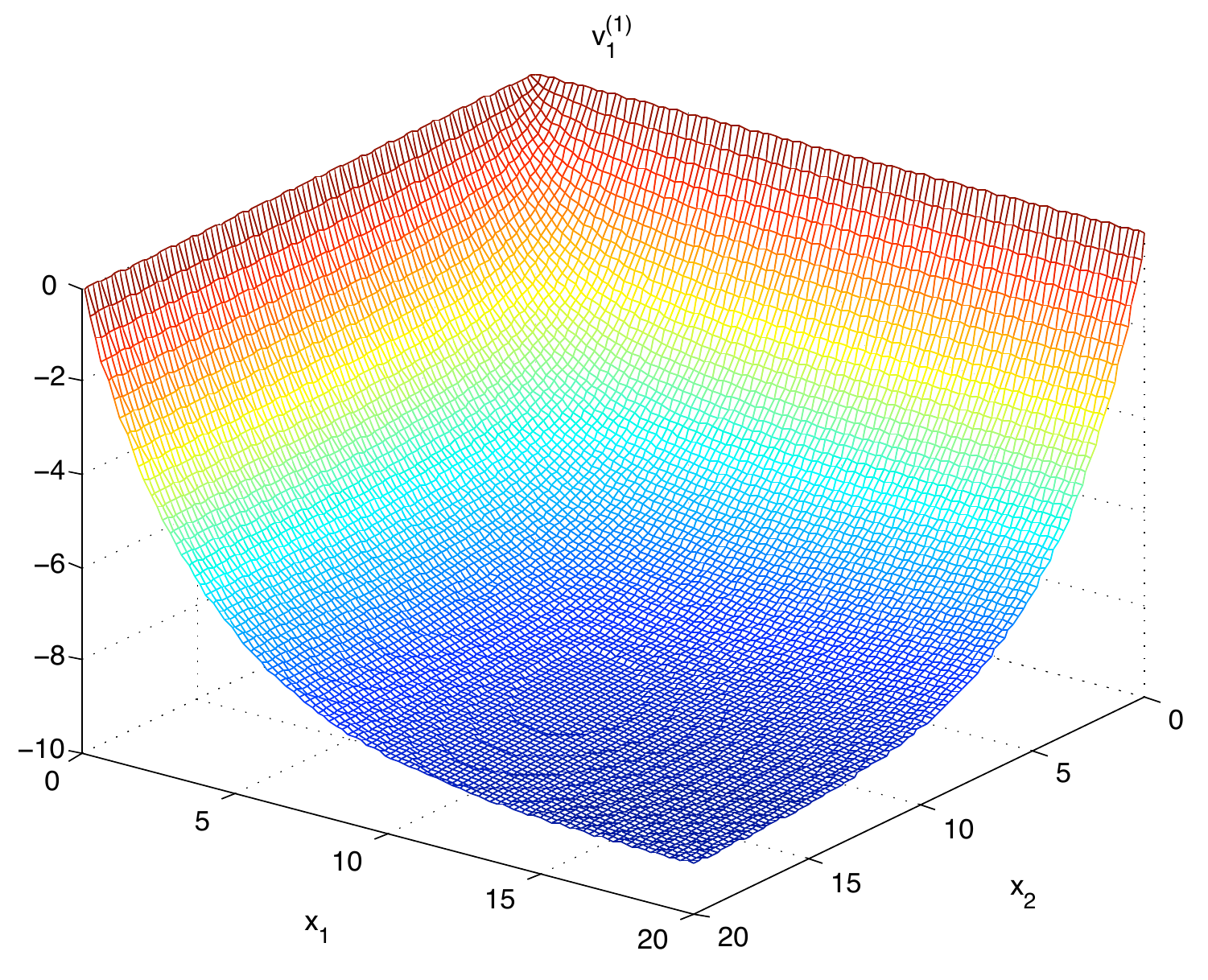}
   \label{fig:v11_surface}
   }
   \subfigure[$v_1^{(2)}$: Surface]{
  \includegraphics[width=3in]{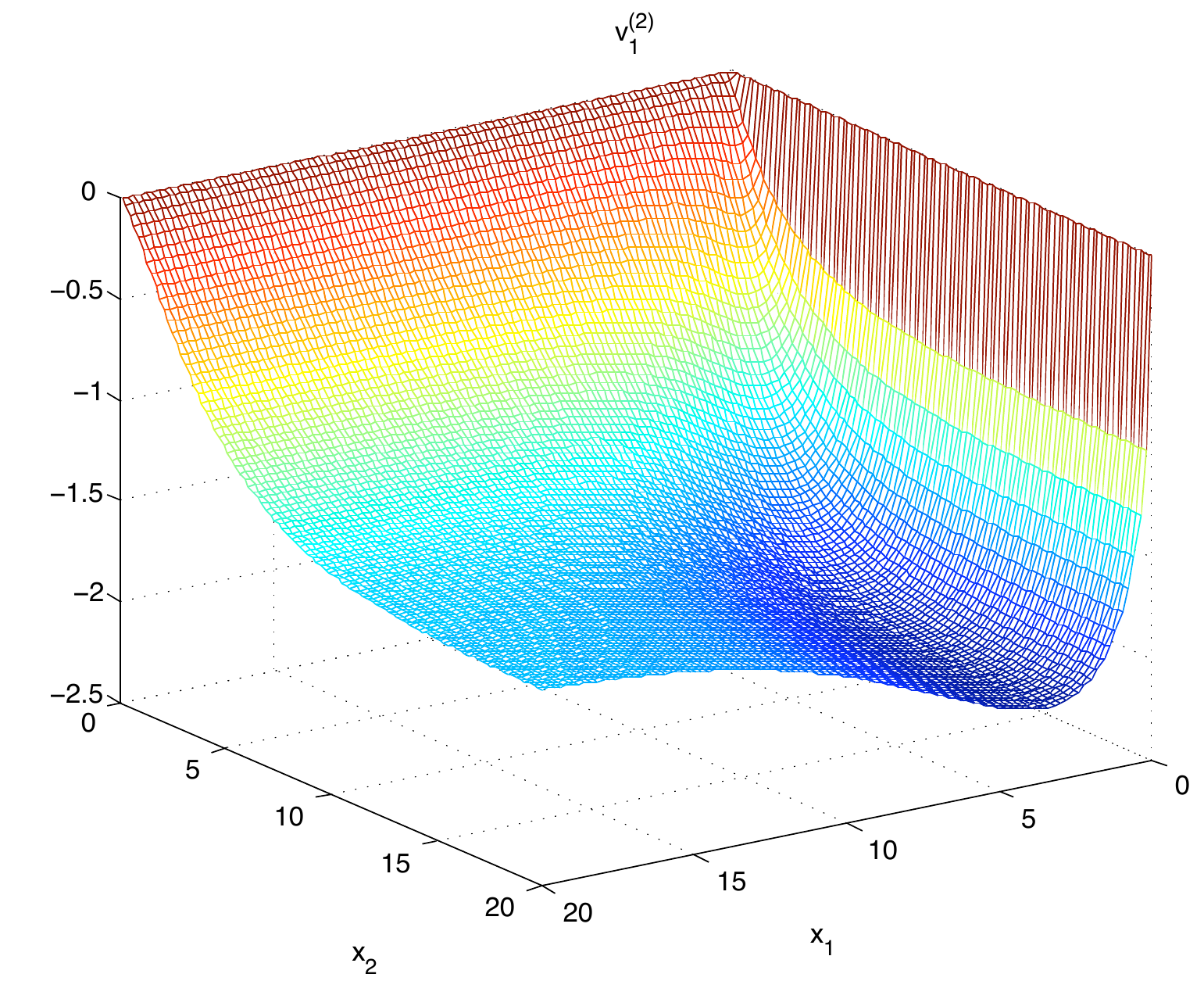}
   \label{fig:v12_surface}
   }
   \caption{$v_1^{(1)}$ and $v_1^{(2)}$. First- and second-order value function corrections.}
   \label{fig:v1_and_v2}
\end{figure}

We plot $v_1^{(2)}$ in Figure \ref{fig:v1_and_v2}\subref{fig:v12_surface}. It is again negative everywhere. Hence, the second-order correction serves to further decrease the value of the game from just the first-order approximation. We plot in Figure \ref{fig:v12minusv22} the difference between $v_1^{(2)}$ and $v_2^{(2)}$. We do not make the same comparison for $v_i^{(1)}$ because $v_1^{(1)} - v_2^{(1)} \equiv 0$.
We see from this figure that the sign of the difference is equal to the sign of $x_1 - x_2$. Consider the situation in which Firm 1 has larger lifetime capacity, $x_1 > x_2$. In the absence of competition, he has a larger value function, $v_M(x_1) > v_M(x_2)$. Introducing competition lowers his value by less than it lowers the value of the smaller firm. Therefore, competition serves to enhance the advantage of the larger firm. Of course, when two firms both have large remaining capacities, the inequality between firms does not have that much importance. It is only when one or both firms have small amounts of capacity remaining that inequalities across firms become magnified.
\begin{figure}[htbp]
   \centering
\includegraphics[width=3in]{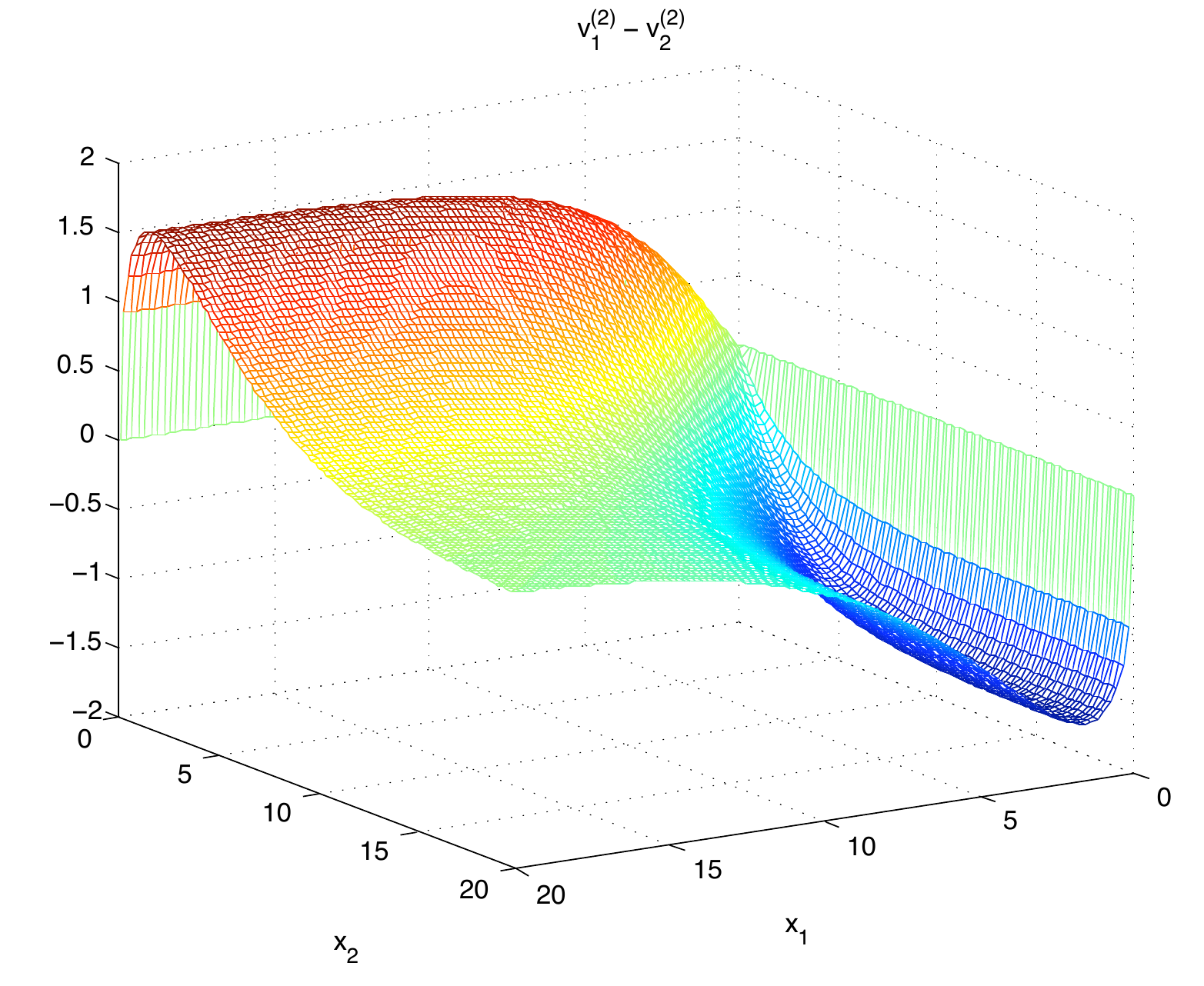}
    \caption{$v_1^{(2)} - v_2^{(2)}$. Difference between Firm 1 and Firm 2 second-order value function corrections.}
       \label{fig:v12minusv22}
   \end{figure}

For the remaining figures of this section we use the expansion up to order $\gamma$. We plot in Figure \ref{fig:priceAndProfit}\subref{fig:price1_x1_10_func_x2} the equilibrium price strategy of Firm 1 as a function of $x_2$ for various values of $\gamma$ using the first two terms of \eqnref{eqn:exp1}. For $\gamma=0$, $\bar p_1^{\star}$ does not change with $x_2$ as the goods are independent. For increasing $\gamma$, as we would expect, the prices decrease at a faster rate as a function of $x_2$ because $\bar p_1^{\star}$ is more sensitive to the price of firm 2 when $\gamma$ is larger.
\begin{figure}[htbp]
   \centering
   \subfigure[Equilibrium price as a function of $x_2$ for $x_1 = 10$ fixed]{
  \includegraphics[width=3in]{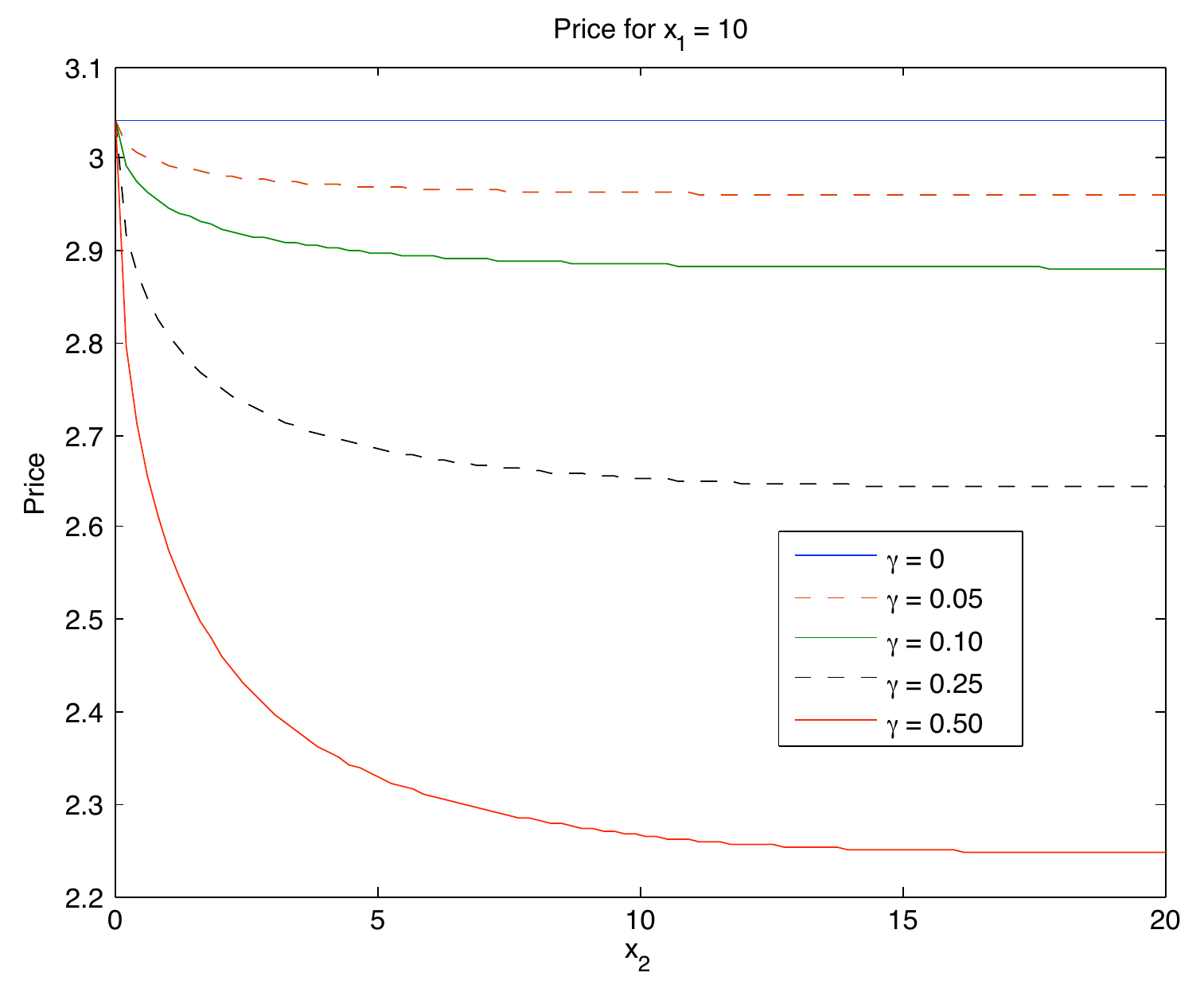}
   \label{fig:price1_x1_10_func_x2}
   }
      \subfigure[$\bar p_1^{\star}\times D_1\left(\bar p_1^{\star},\bar p_2^{\star}\right)$ as a function of $\gamma$ for $x_1 = 10$]{
   \includegraphics[width=3.1in]{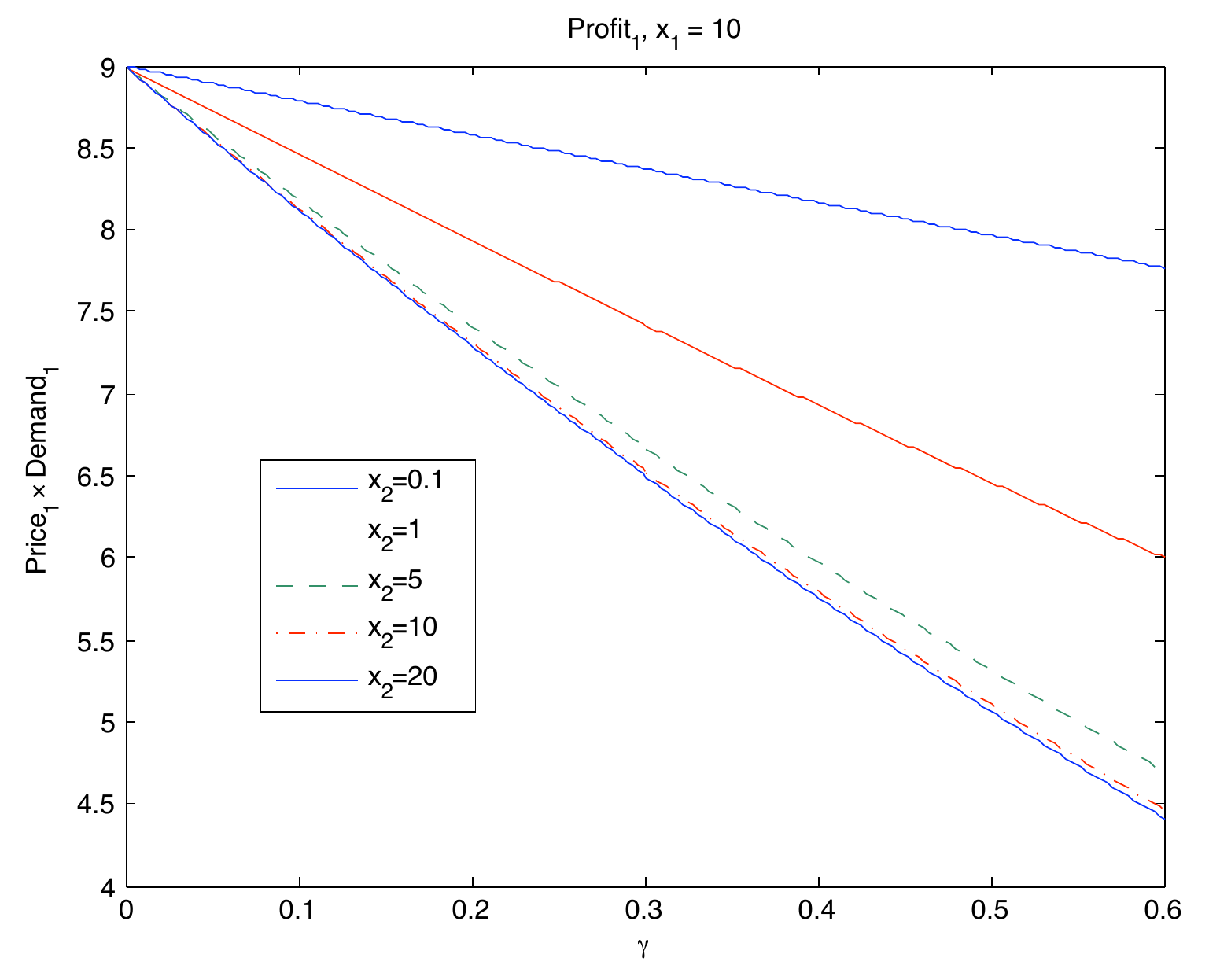}
   \label{fig:profit1_x1_10}
   }
    \caption{Equilibrium Price and Profit for Firm 1}
    \label{fig:priceAndProfit}
\end{figure}
In Figure \ref{fig:priceAndProfit}\subref{fig:profit1_x1_10}, we plot the instantaneous profit for Firm 1, $\bar p_1^{\star}D_1(\bar p_1^{\star},\bar p_2^{\star})$. Each line shows the impact at a fixed $(x_1,x_2)$ of increasing $\gamma$, and we see from the slopes that the effect on reducing instantaneous profit is greater when $x_2$ is larger. Hence, competition, as measured by $\gamma$, only has bite when the competitor has comparable lifetime capacity levels.

Thus, when one considers the effect of competition on profit and prices, there are two main pieces of information one must take into account. The first is the degree of substitutability between goods. Fundamentally this means how similar is your good to your competitor's. For example, one expects $\gamma$ to be large in the market for gasoline, because gas at one station is essentially the same as gas at a station across the street. The goods are not perfect substitutes because of travel costs, brand loyalty, and a host of other reasons. However, in the market for CDs, one expects a very low degree of substitutability, because an individual artist's music is typically highly differentiated from that of another artist, even within a specific genre. It is reasonable to assume in such a market that $\gamma$ is quite low.

The second main piece of information is how credible is your competition. In this model, the proxy for a credible competitor is their level of lifetime capacity. If your competitor has very little lifetime capacity, then it stands to reason that regardless of how similar their good is to your own good, you will be a monopoly in short order. In contrast, if your competitor has a large amount of lifetime capacity, then they are a very credible threat to your business. As such, you must take such them very seriously even if their good is highly differentiated (but still substitutable) with your good. This can be seen in Figure \ref{fig:priceAndProfit}\subref{fig:profit1_x1_10} as even with $\gamma=0.1$, when $x_2$ is large relative to $x_1$, the instantaneous profit is much less than when $x_2$ is relatively small.

In Figure \ref{fig:pathOfGame_multGam}, we plot the solution to
$\frac{dx_1}{dt} = -D_1\left(\bar p_1^{\star}(\mathbf{x}),\bar
p_2^{\star}(\mathbf{x})\right)$, where we have used our expansion of
order $\gamma$. We present the path of $x_1(t)$ over time for
various different values of $\gamma$ starting from $x_1(0) = x_2(0)
= 10$. We see that as $\gamma$ increases, the time of the game
increases. That is, it takes more time for Firm 1 to deplete their
lifetime capacity when the degree of substitutability is greater.
\begin{figure}[htbp]
   \centering
   \includegraphics[width=3.5in]{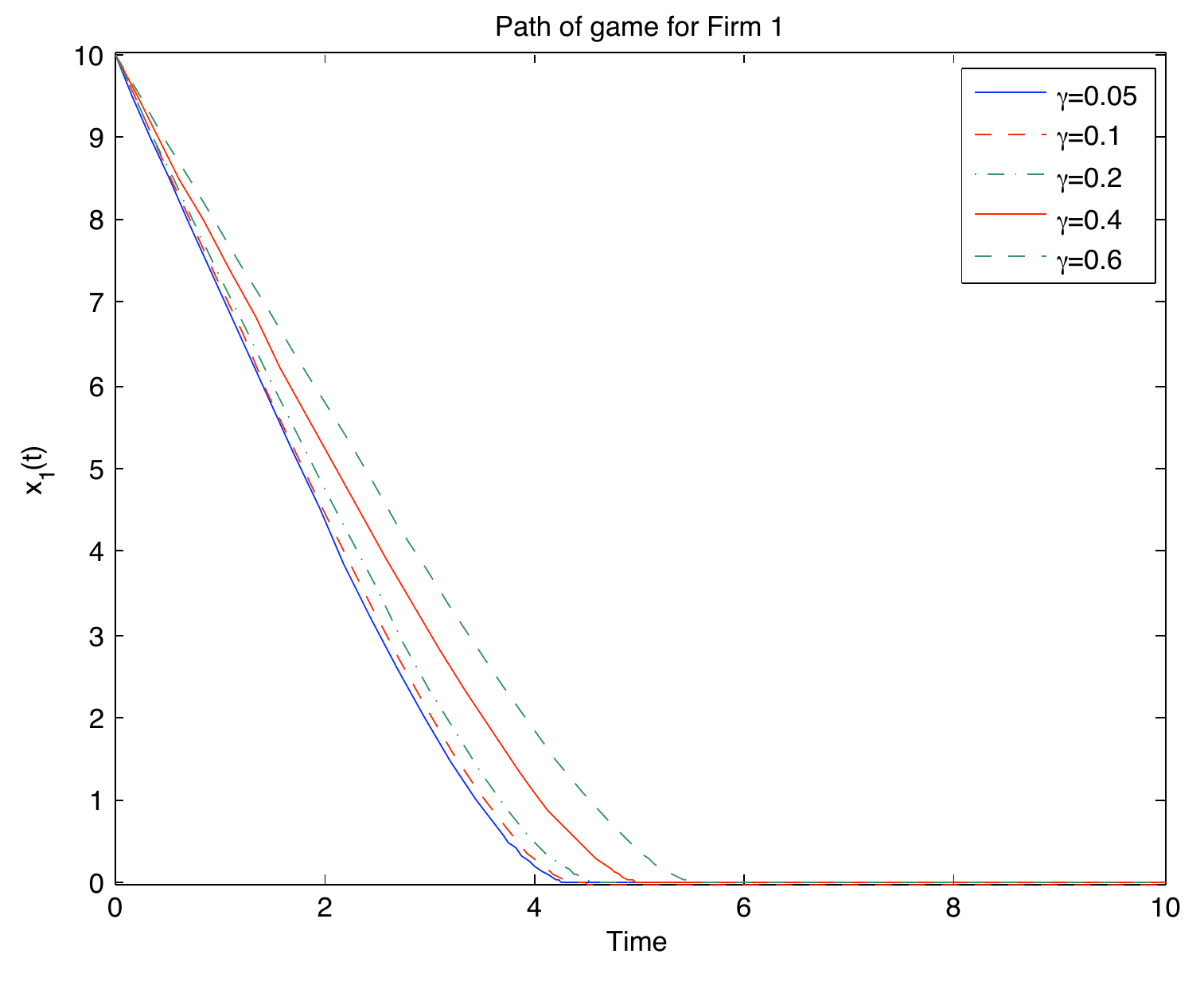}
   \caption{Path of Firm 1 Lifetime Capacity over time with $x_1(0) = x_2(0) = 10$ for various values of $\gamma$}
   \label{fig:pathOfGame_multGam}
\end{figure}
We can see this also by looking at the path of both demand and prices over the time of the game. These are plotted in Figures \ref{fig:pathOfGame_multGam_DandP} \subref{fig:pathOfGame_multGam_demand} and \subref{fig:pathOfGame_multGam_price}, respectively. We see that, for an individual firm, increasing the value of $\gamma$ drives down the price, while the demand first decreases and then increases. Common across all levels of $\gamma$, we see that as the capacity of the firm diminishes over time, the price increases while demand decreases.
\begin{figure}[htbp]
   \centering
   \subfigure[Demand]{
   \includegraphics[width=3.1in]{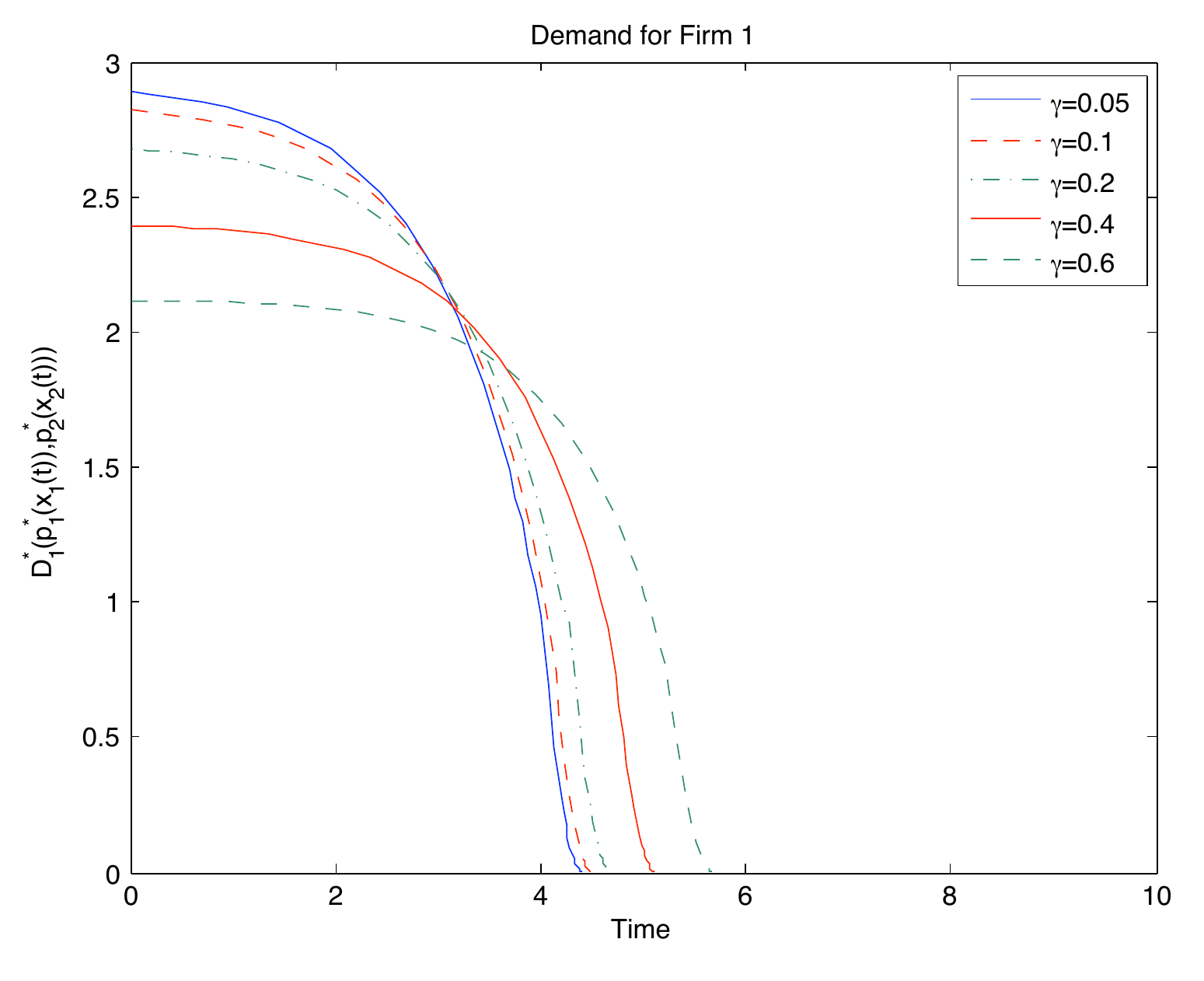}
   \label{fig:pathOfGame_multGam_demand}
   }
   \subfigure[Price]{
   \includegraphics[width=3.1in]{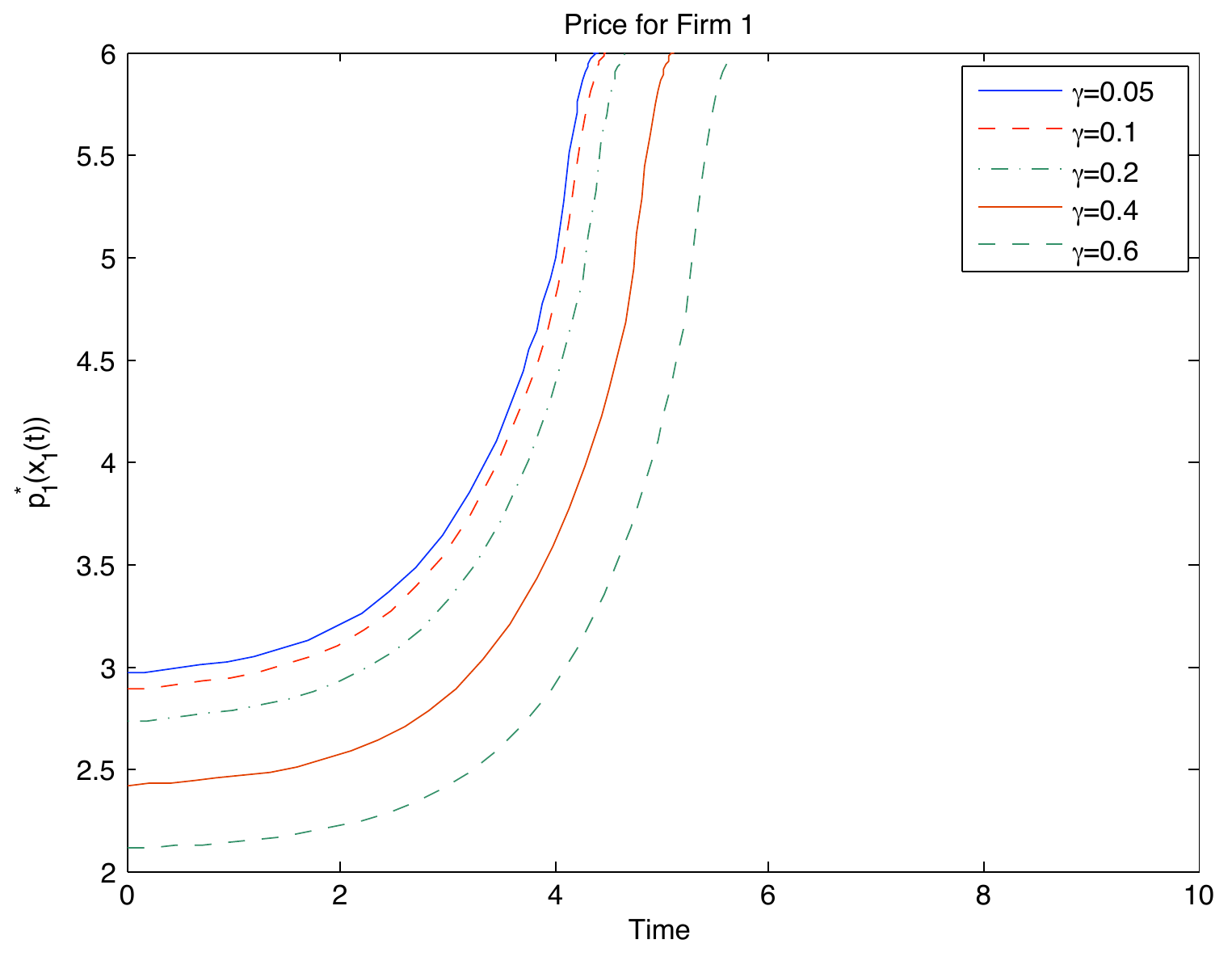}
   \label{fig:pathOfGame_multGam_price}
   }
   \caption{Path of Firm 1 Price and Demand over time with $x_1(0) = x_2(0) = 10$ for various values of $\gamma$}
   \label{fig:pathOfGame_multGam_DandP}
\end{figure}

Finally, we remark that one can confirm that the resulting demands from the asymptotic approximation for both firms are strictly positive in equilibrium. Therefore, our previous assumption on strictly positive demands is justified at least for $\gamma$ small enough.

\section{Numerical Analysis}\label{sec:numericalAnalysis}
We have been able to capture many of the qualitative features of the model analytically in the case of deterministic demand using the asymptotic expansions of the previous section. However, in order to fully analyze the stochastic model we solve the full PDE system numerically. With these solutions, we then simulate paths of the game.

\subsection{Finite Difference Solution of the PDE System}\label{sec:FDsolnPDE}
We employ fully implicit finite differences to find solutions $V_1$ and $V_2$ to the system in \eqnref{eq:PDEsystem2}. We can then use finite difference approximations to the derivatives of these value functions to obtain optimal price strategy functions and the resulting optimal demand functions. We use the parameter values $\alpha = 6$, $\beta = 1$, $r=1$, $\sigma_1 = \sigma_2 = 0.6$ and $\rho = 0.1$. The demands remain strictly positive for this choice of parameters and therefore, in this case, both firms are active participants in equilibrium.

For given levels of capacity $x_1$ and $x_2$, we define $\theta \triangleq \tan^{-1}\left(x_2/x_1\right)$.
This is the angle in $(x_1,x_2)$ space that corresponds to the given capacity levels. We plot in Figure \ref{thetaFigs} \subref{priceTheta_1}, the price of Firm 1 as a function of this $\theta$ for several values of $\gamma$. We also plot the optimal demand as a function of $\theta$ in Figure \ref{thetaFigs} \subref{demandTheta_1}.
\begin{figure}[htp]
   \centering
   \subfigure[Price]{
    \includegraphics[width=3.1in]{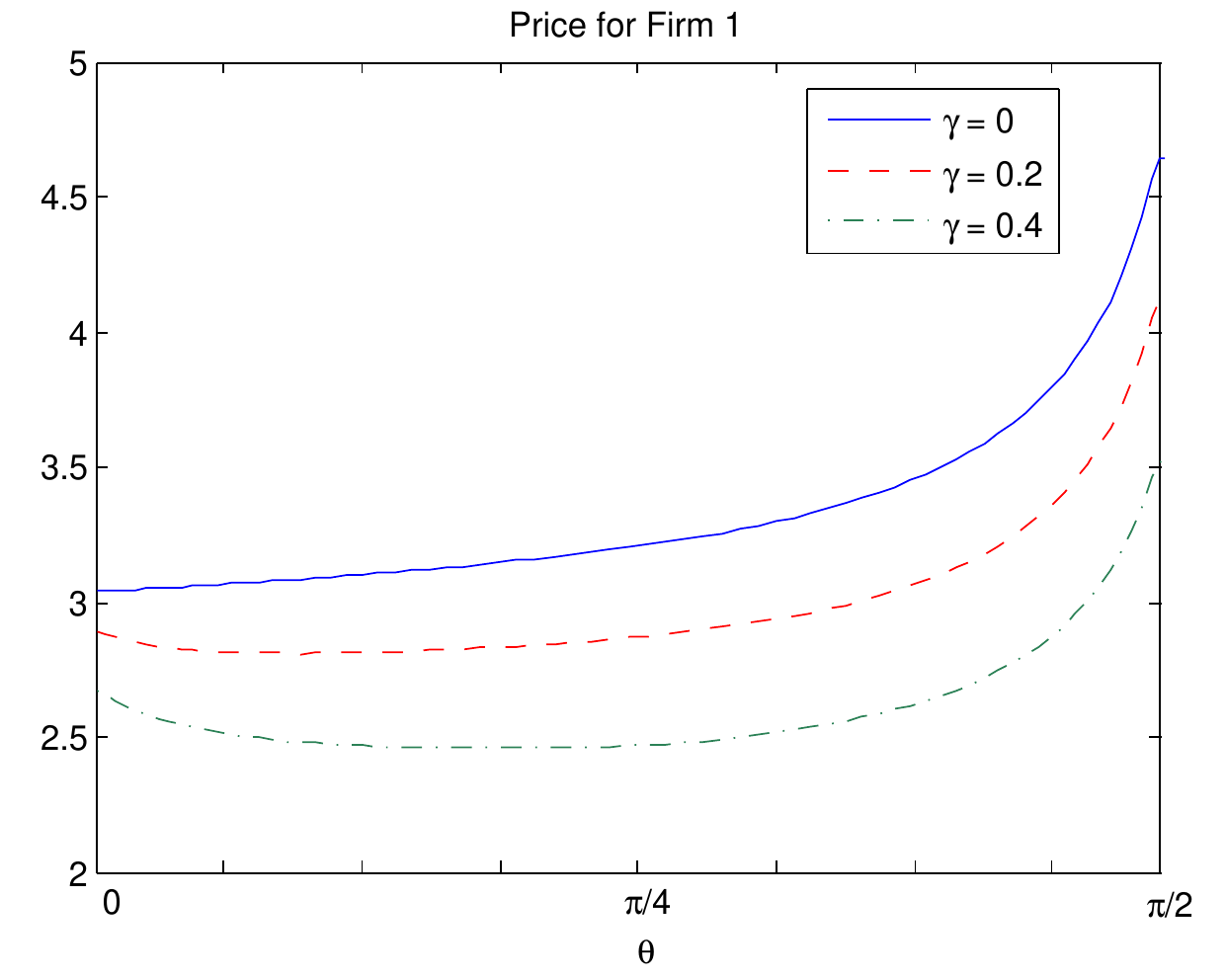}
   \label{priceTheta_1}
   }
   \subfigure[Demand]{
    \includegraphics[width=3.1in]{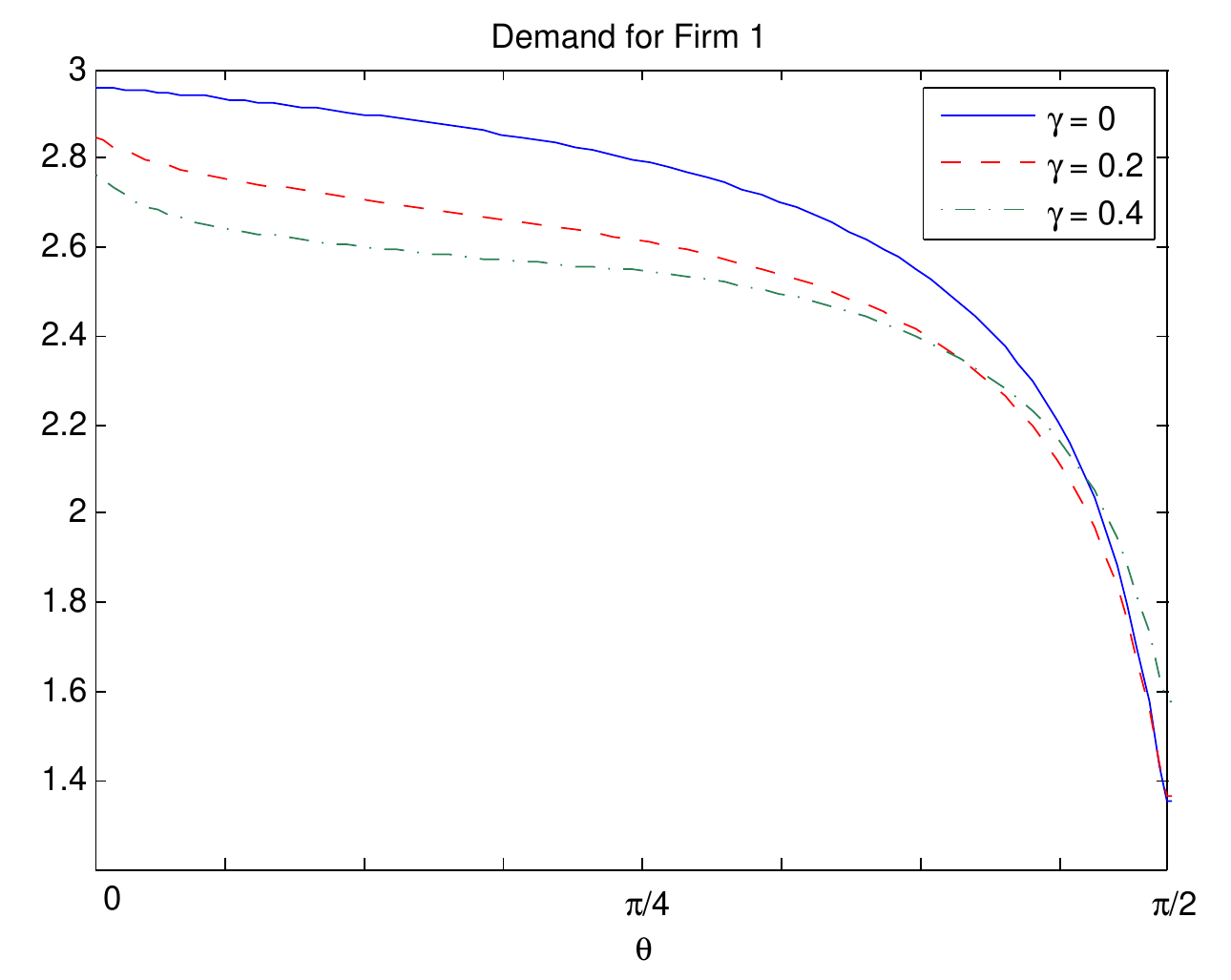}
   \label{demandTheta_1}
   }
    \caption{Price and Demand for Firm 1 as a function of $\theta$ for various values of $\gamma$}
    \label{thetaFigs}
\end{figure}
These figures allow us to analyze price and demand effects for a fixed level of overall capacity in the market. At one end of the spectrum, $\theta = 0$, we have that Firm 1 controls all of the capacity in the market. On the other end, $\theta = \pi/2$, Firm 2 controls all of the capacity. The price of Firm 1 decreases initially as $\theta$ increases, provided there is competition in the market, i.e. $\gamma > 0$. However, the price starts to increase once $\theta$ moves beyond $\pi/4$ and Firm 1 holds the minority share of the available capacity in the market. Furthermore, price decreases at a greater rate for higher levels of $\gamma$. This implies that the optimal price strategies result in the lowest prices when substitutability is high, i.e. $\gamma$ is large, and when all firms in a market are of the same relative size, i.e. $\theta$ is close to $\pi/4$. Increasing the number of firms in a market, thereby increasing the level of competition, does not always have a large effect on consumer welfare. This is because if an additional firm enters a market with a small capacity when there already exist large capacity firms, then essentially this means they are not credible competitors to the large firms and therefore prices do not necessarily have to adjust by a large amount to account for this additional competition. On the other hand, we see from these figures, for a given fixed level of total market capacity, consumers are always better off if the capacity is spread evenly across firms, provided the firms produce substitute goods (i.e. $\gamma > 0$). Consumers therefore face the best possible situation when there exist multiple firms in a market whose goods have a high degree of substitutability and whose relative sizes are the close to the same.

Figures \ref{contourP} \subref{contourP_0} and \subref{contourP_0p4} are the contours of the optimal price strategy for Firm 1 for two different values of $\gamma$. The strategy changes from charging a constant price as a function of $x_2$ when $\gamma = 0$, to charging a price that is decreasing in $x_2$ for $\gamma = 0.4$. This is as expected because prices should reflect the credibility of one's competitor when that competitor produces a good which is substitutable with one's own good. We see that although prices are decreasing in both $x_1$ and $x_2$, they decrease at a faster rate with respect to $x_1$.

\begin{figure}[htp]
    \centering
    \subfigure[$\gamma = 0$]{
    \includegraphics[width=3.1in]{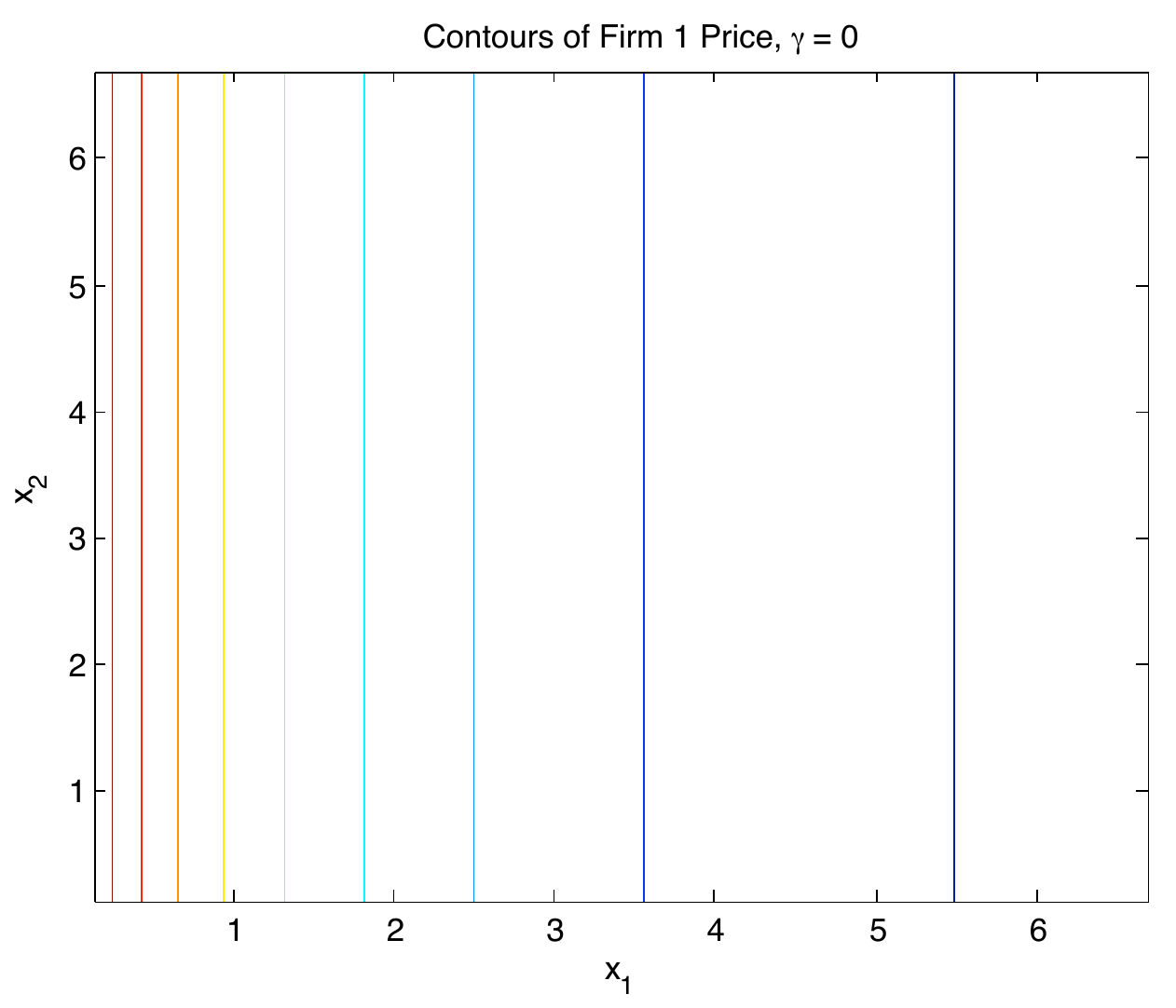}
    \label{contourP_0}
    }
    \subfigure[$\gamma = 0.4$]{
    \includegraphics[width=3.1in]{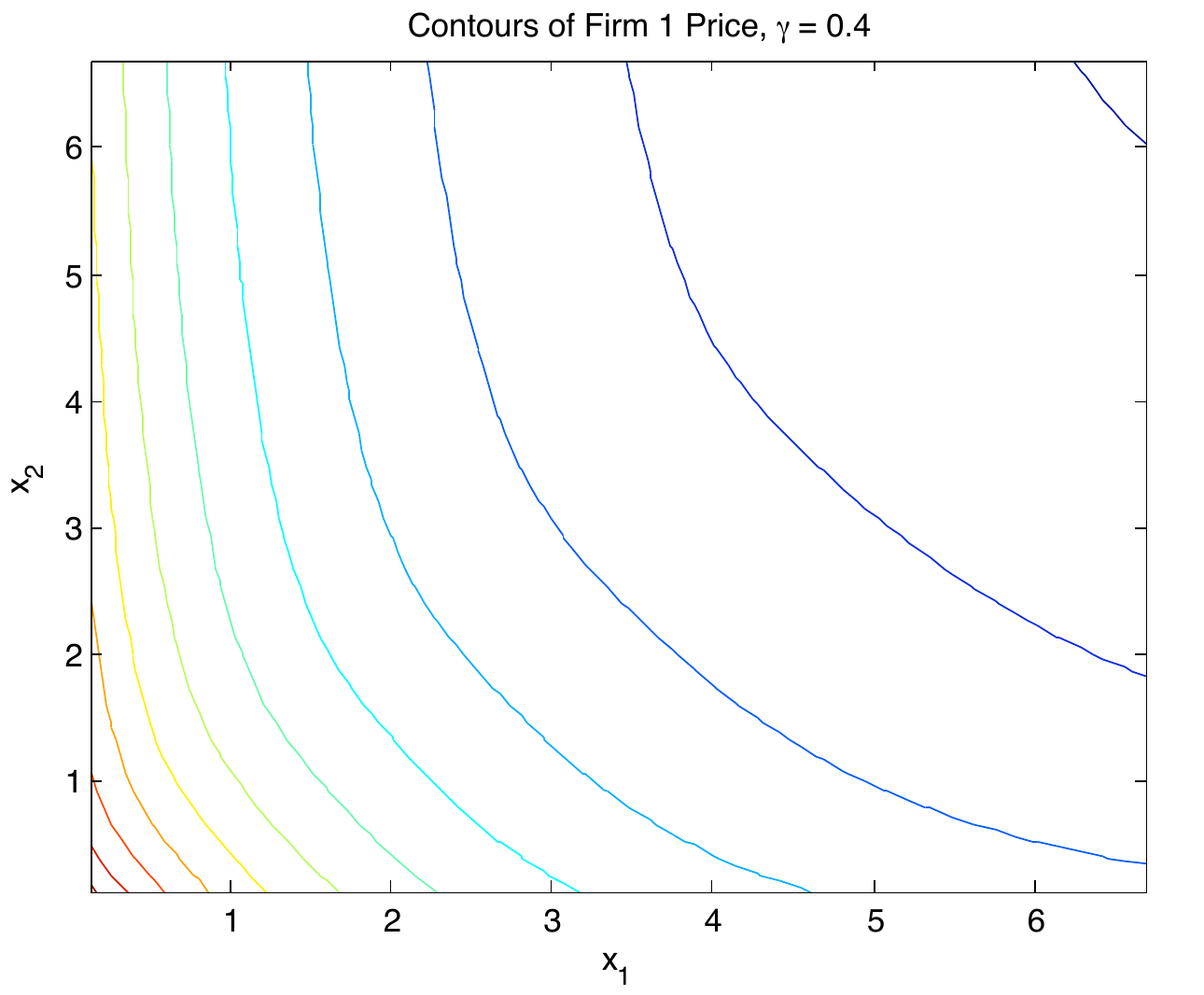}
    \label{contourP_0p4}
    }
    \caption{Contours of Firm 1 Price Strategy for $\gamma=0$ and $0.4$.}
    \label{contourP}
\end{figure}

\subsection{Game Simulations}
We can use our numerical solution of the value functions to obtain numerical versions of the equilibrium prices, $\bar p_i^{\star}(\mathbf{x}(t))$. These in turn can be used to obtain numerical versions of the equilibrium capacity trajectories over time by making use of \eqnref{eq:inventoryDynamics}, i.e.
\begin{equation*}
dx_i(t) = -D_i\left(\bar p_1^{\star}(x_1(t),x_2(t)),\bar p_2^{\star}(x_1(t),x_2(t))\right)~dt + \sigma_i~dW_i(t), \hspace{0.5cm} i=1,2.
\end{equation*}
We simulate paths of the Brownian motions in order to obtain paths of the game over time. We present an example of these paths in Figure \ref{fig:paths}\subref{fig:path_3} for $\gamma = 0.4$. This is just one example but it displays many features of the game. Here we started both firms with an initial capacity of $x_1(0) = x_2(0) = 10$. Over time these decrease, but because of the stochastic component of demand, we see that at any given moment in time, we cannot say with certainty which firm will necessarily have a higher level of capacity. In this example, initially Firm 2 gains an advantage, however we see as time progresses that Firm 2 runs out of capacity before Firm 1, and hence for a period of time Firm 1 has a monopoly.
\begin{figure}[htbp]
   \centering
  \subfigure[Path of Stochastic Game for both firms with $\gamma = 0.3$]{
   \includegraphics[width=3.1in]{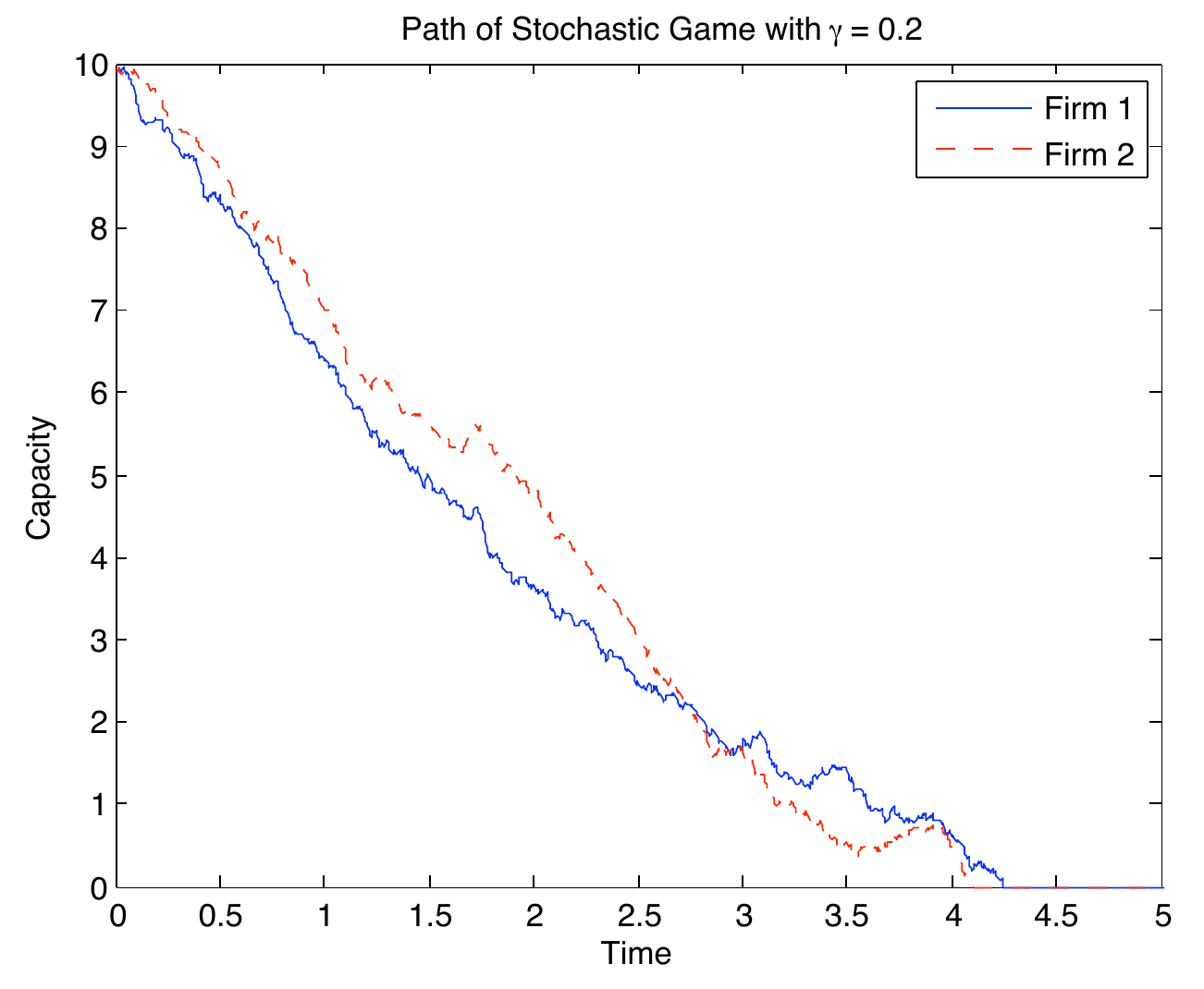}
     \label{fig:path_3}
     }
        \subfigure[Paths of game for firm 1 for varying $\gamma$]{
    \includegraphics[width=3.1in]{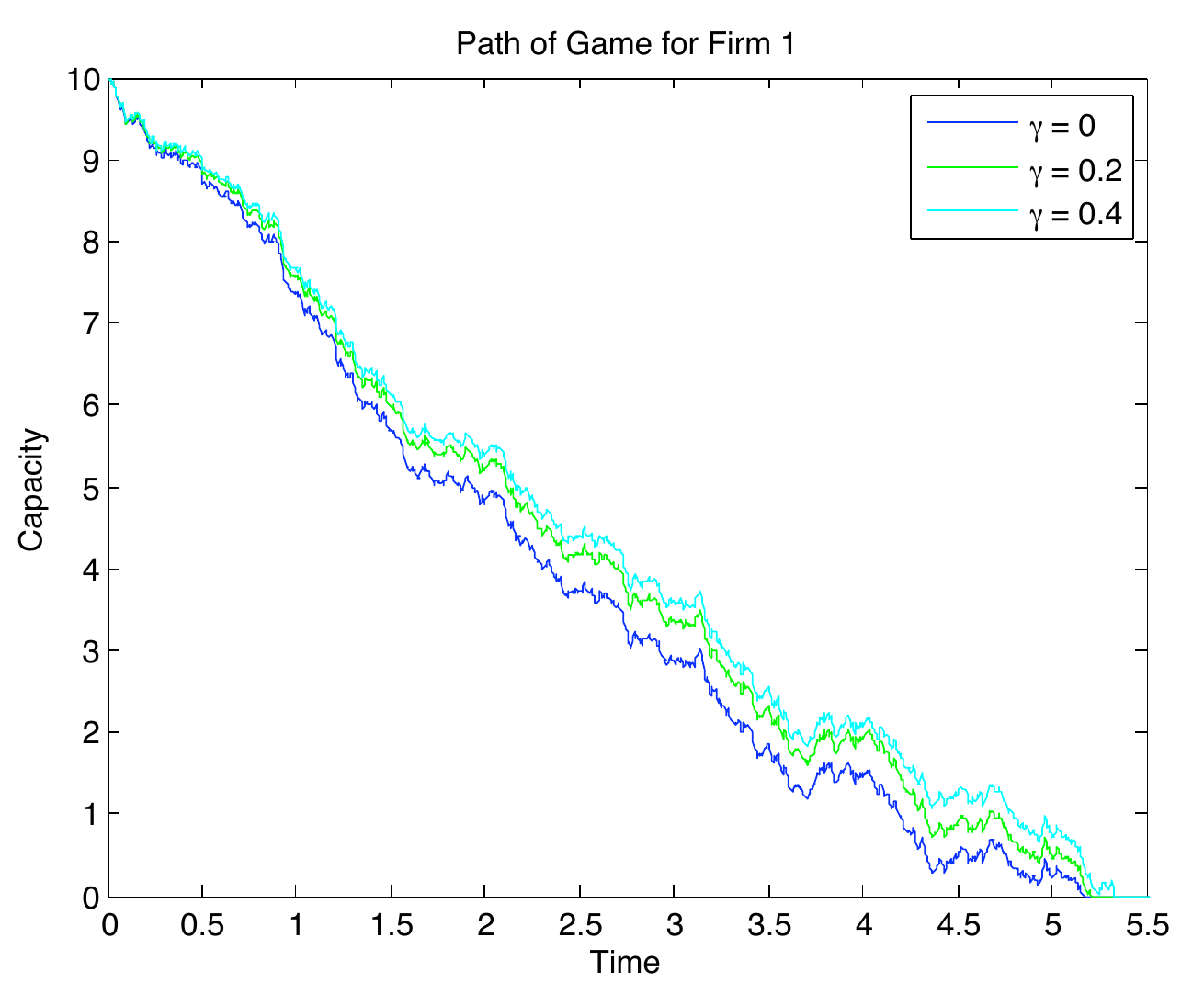}
    \label{fig:path_1}
    }
    \caption{Path of Stochastic Game}
    \label{fig:paths}
\end{figure}
For a fixed realization of the Brownian path, we plot in Figure \ref{fig:paths}\subref{fig:path_1}, the path of the capacity for Firm 1 for various values of $\gamma$. We see that as $\gamma$ increases, the time until the capacity of Firm 1 is exhausted, also appears to increase. This effect, where increased competition prolongs the lifetime of firms, is consistent with what we saw in Figure \ref{fig:pathOfGame_multGam} in the deterministic game.

\section{Conclusion}\label{sec:conclusion}
We have studied nonzero-sum stochastic differential games arising from Bertrand competitions, in particular, the case of a duopoly with linear demand functions. By considering the case where there is a small degree of substitutability between the firms' goods, we are able to construct an asymptotic approximation that captures many of the qualitative features of the ordinary differential game. Numerical solutions further provide insight into the stochastic case and where there is a higher degree of substitutability. In our study of the dynamic game, we concentrated on the case where both firms have positive demand in equilibrium. As we saw in the static game, it is possible for there to arise cases where this does not occur. However, the study of such cases remains an interesting open question in the analysis of these dynamic Bertrand games.

These tools allow us to quantify the effects of substitutability and relative firm size on prices, demands and profits. In particular, we find that consumers benefit the most when a market is structured with many firms of the same relative size producing highly substitutable goods. However, a large degree of substitutability does
not always lead to large drops in price, for example when two firms have a large difference in their size. That is, consumers benefit the most when firms are competing with other firms whom they deem to be credible threats. A competitor is a credible threat only if they have capacity large enough to match their competitors. For example, the existence of a very small coffee shop does not greatly affect the pricing decisions of Starbucks.

It is of interest to extend the analysis here to markets with nonlinear demand systems as well as to random demand environments. By the latter, we mean that baseline demand and substitutability may vary stochastically over time with economic conditions or consumer tastes.
For example, in the linear demand system \eqnref{eq:dNDefine}, there may be fluctuations in the intercept parameter $A$, which is a measure of the general level of demand due to business cycles and recessions, or in $C/B$, which is the measure of substitutability, as certain brands fall out of fashion, for example what Toyota is experiencing in 2010. Finally, an important related direction is to connect, in a probabilistic way, prices to the mechanism of customer arrival, for example, by making the arrival intensity a function of prices. Such models have been used in a variety of applications for the single-player price setting problem. However, their use in competitive markets leads to multi-player stochastic differential games with jumps, and is a direction we are currently pursuing.

\begin{appendix}
\section{Proof of Proposition \ref{prop:existenceOfChokePrice}}\label{appendix:proof}
We have assumed that $\hat{p}_N$ exists. We show that this implies $\hat{p}_{N-1}$ exists. The result for all $n$ will then follow by induction. Fix $p_1,\ldots,p_{N-2}$ and denote this vector of prices by $\rho$. Let $p_{N-1}$ be an arbitrary price. We denote $\hat{p}_N(\rho,p_{N-1})$ by $\tilde{p}$. We then have two cases, depending on $p_{N-1}$, under which we wish to show $D_{N-1}^{N-1}\left(\rho,p_{N-1}\right) < 0$. The first case is $p_{N-1} > \hat{p}_N(\rho,\tilde{p})$, which implies
\begin{equation*}
D^{N-1}_{N-1}\left(\rho,p_{N-1}\right) = D^N_{N-1}\left(\rho,p_{N-1},\tilde{p}\right) =  D^N_N\left(\rho,\tilde{p},p_{N-1}\right) < D^N_N\left(\rho,\tilde{p},\hat{p}\left(\rho,\tilde{p}\right)\right) = 0.
\end{equation*}
The second case we must consider is the reverse inequality, i.e. $p_{N-1} < \hat{p}_N(\rho,\tilde{p})$. This implies
\begin{equation*}
0 = D^N_N\left(\rho,p_{N-1},\tilde{p}\right) >  D^N_N\left(\rho,\hat{p}\left(\rho,\tilde{p}\right),\tilde{p}\right) = D^N_{N-1}\left(\rho,\tilde{p},\hat{p}\left(\rho,\tilde{p}\right)\right) = D^{N-1}_{N-1}\left(\rho,\tilde{p}\right).
\end{equation*}
Hence, regardless of the relative size of $p_{N-1}$ and $\hat{p}_N\left(\rho,\tilde{p}\right)$, we have that $D^{N-1}_{N-1}(p)$ is negative for some vector of prices $p$. We have ignored the case $p_{N-1} = \hat{p}_N(\rho,\tilde{p})$, because this would change the above inequalities to equalities and we would have the choke price we are looking for.

We know by Assumption \ref{assump:downwardSlopingDemand} that $D^{N-1}_{N-1}$ is a decreasing function of $p_{N-1}$. We have shown above that there exists some vector of prices $p$ such that $D^{N-1}_{N-1}(p) < 0$. We thus need only show that there exists some vector of prices $p$ where $D^{N-1}_{N-1}(p) > 0$. This will establish the existence of the choke price.

Fix some vector of prices $\bar{p} = \left(p_1,\ldots,p_{N-2},0\right)$, where $p_1,\ldots,p_{N-2}$ are arbitrary prices. Denote by $\underbar{0}$ the vector in $\mathbb{R}^{N-1}$ of all zeros. Recall that $D^N_N\left(\underbar{0},0\right) > D^N_N\left(\underbar{0},\hat{p}_N\left(\underbar{0}\right)\right) = 0$. This implies $\hat{p}_N\left(\underbar{0}\right) > 0$. We then have
\begin{equation*}
D^{N-1}_{N-1}\left(\bar{p}\right)  = D^N_{N-1}\left(\bar{p},\hat{p}_N\left(\bar{p}\right)\right)  \geq D^N_{N-1}\left(\bar{p},\hat{p}_N\left(\underbar{0}\right)\right)  > D^N_{N-1}\left(\bar{p},0\right) \geq D^N_{N-1}\left(\underbar{0},0\right)  = D^N_N\left(\underbar{0},0\right)  > 0.
\end{equation*}
Hence, for the vector $\bar{p}$ we have that $D^{N-1}_{N-1}(\bar{p})$ is positive. Therefore, there must exist some $\hat{p}_{N-1}$ for every set of prices $p_1,\ldots,p_{N-2}$ such that $D^{N-1}_{N-1}\left(p_1,\ldots,p_{N-2},\hat{p}_{N-1}\right) = 0$.

\section{Second-order approximation}\label{appendix:2ndOrderPDE}
We give the solution of the PDEs \eqnref{eqn:expansionPDE3} for $v_i^{(2)}$.
Using the solution to $v_i^{(1)}$ found in Proposition \ref{prop:solnOfPDE1}, making the same change of variables $(\xi,\eta) =
  \left(Q(x_1),Q(x_2)\right)$ and solving leads to
\begin{eqnarray}
v_1^{(2)} &=& \frac{\alpha^2}{8\beta^3}e^{-rQ(x_2)}\left[-\frac{Q(x_2)^2}{2}\left(\frac{2re^{rQ(x_2)}}{e^{rQ(x_2)}-1} + \frac{r\varphi_1}{1-e^{-rQ(x_1)}}\right) - \frac{Q(x_2)(1-\varphi_1)}{1-e^{-rQ(x_1)}} \right.\nonumber\\
& & \left.+ \frac{3}{2r}\left(e^{rQ(x_2)}-1\right) - \frac{\left(1-\varphi_1\right)^2}{2r\left(e^{rQ(x_2)}\left(1-e^{-rQ(x_1)}\right)\right)} + \frac{1-\varphi_1}{2r}\right. \nonumber \\
& & \left. -\frac{3}{r}\left(e^{rQ(x_2)} - \varphi_1^2 e^{-rQ(x_2)} - 1 + \varphi_1^2 - 2\varphi_1 r Q(x_2) \right) \right].\label{eq:v12FinalSolnTop}
\end{eqnarray}
for $x_1 > x_2$ where $\varphi_1 = \exp\left\{-r\left(Q(x_1)-Q(x_2)\right)\right\}$. For $x_2 > x_1$,
\begin{eqnarray}
v_1^{(2)} &=& \frac{\alpha^2}{8\beta^3}e^{-rQ(x_1)}\left[-\frac{Q(x_1)^2}{2}\left(\frac{re^{rQ(x_1)}}{e^{rQ(x_1)}-1} + \frac{2r\varphi_2}{1-e^{-rQ(x_2)}}\right) - \frac{2Q(x_1)(1-\varphi_2)}{1-e^{-rQ(x_2)}}\right.\nonumber\\
& & \left.+ \frac{3}{2r}\left(e^{rQ(x_1)}-1\right)-\frac{(1-\varphi_2)^2}{re^{rQ(x_1)}\left(1-e^{-rQ(x_2)}\right)}+\frac{1-\varphi_2}{r}\right. \nonumber\\
& & \left. -\frac{3}{r}\left(e^{rQ(x_1)} - 2rQ(x_1) - e^{-rQ(x_1)}\right) \right].\label{eq:v12FinalSolnBot}
\end{eqnarray}
where $\varphi_2 = \exp\left\{-r\left(Q(x_2)-Q(x_1)\right)\right\}$.
We omit the details of this lengthy calculation.

\section{$N$-player Stochastic Differential Game with Linear Demands}\label{append:sec:NplayerPDE}
Consider the $N$-player dynamic Bertrand game under the linear demand system introduced in Section \ref{sec:diffGame}. Within this section we explicitly assume that all firms participate in equilibrium, i.e. all firms receive positive demand. In the language of Section \ref{sec:StaticGame}, this would mean that the resulting dynamic equilibrium is of Type \ref{NEtype1}. Let $V_i(\mathbf{x})$ denote the value functions defined by the $N$-player analog of \eqnref{eq:Vi}. Then, the associated PDE system, analog of \eqnref{eq:PDEsystem1}, is
\begin{equation}
\mathcal{L}V_i + \sup_{p_i\geq 0}\left\{ - \sum_{k=1}^{N}D^N_k(p)\pderiv{V_i}{x_k} + p_iD^N_i(p)\right\} - rV_i = 0, ~~ i=1,\ldots,N,
\end{equation}
where we explicitly denote the dependence of the demand functions on $N$ to indicate the size of the vector $p$ in the argument. Here,
\begin{equation}
\mathcal{L} = \frac{1}{2}\sum_{i=1}^N\sum_{j=1}^N \Sigma_{ij}\pcderivs{V_i}{x_i}{x_j},
\end{equation}
where $\left(\Sigma_{ij}\right)_{i,j=1,\ldots,N}$ is the covariance matrix between the Brownian motions in \eqnref{eq:inventoryDynamics}.

For a fixed $i$, we note that we can write
\begin{eqnarray}
D^N_j(p) &=& -\frac{C}{B}D^N_i(p) + A\left(1+\frac{C}{B}\right) - B\left(1-\frac{C^2}{B^2}\right)p_j + C\left(1+\frac{C}{B}\right)\underset{k\neq i}{\sum_{k\neq j}}p_k \nonumber \\
&=& -\frac{C}{B}D^N_i(p) + D^{N-1}_{\pi_i(j)}\left(p_{-i}\right),
\end{eqnarray}
where $p_{-i} = p \backslash \{p_i\}$ and
\begin{equation}
\pi_i(k) = \left\{\begin{array}{ccl} k, & & k < i \\ k-1, & & k > i.\end{array}\right.
\end{equation}
This decomposition means that one can represent a firm's demand function at the level $N$ as a linear combination of another given firm's demand function at the level $N$ and their own demand function at the level $N-1$ with the given firm being removed. This is a consequence of the consistency of demand functions and the existence of choke prices.

Hence, using this decomposition, we have, for $i=1,\ldots,N$
\begin{equation}
\mathcal{L}V_i  - \sum_{k\neq i}^{N}D^{N-1}_{\pi_i(k)}(p_{-i})\pderiv{V_i}{x_k} + \sup_{p_i\geq 0}\left\{ D^N_i(p)\left(p_i-\pderiv{V_i}{x_i} + \frac{C}{B}\sum_{k\neq i}\pderiv{V_i}{x_k}\right)\right\} - rV_i = 0.
\end{equation}
Therefore, we can use the results of the static Bertrand game to write the PDEs as
\begin{equation}
\mathcal{L}V_i - \sum_{k\neq i}^{N}D^{N-1}_{\pi_i(k)}(p^{\star}_{-i})\pderiv{V_i}{x_k} + G_i^N\left(s_1,\ldots,s_N\right) = rV_i
\end{equation}
where the prices $p^{\star}$ are the solutions of the $N$-player static game with costs
\begin{equation}
S_i(\mathbf{x}) = \pderiv{V_i}{x_i} - \frac{C}{B}\sum_{k\neq i}\pderiv{V_i}{x_k}, \qquad i=1,\ldots,N
\end{equation}
and with corresponding profit functions $G_i^N\left(s_1,\ldots,s_N\right) = D^N_i(p^{\star})\left(p^{\star}_i - s_i\right)$.

The boundary conditions for this PDE depends on the PDEs that result by considering a market with $N-1$ firms. That is, on any edge of an orthant where $x_i = 0$ and $x_j > 0$ for all $j\neq i$, we have $V_i \equiv 0$ and that $V_j$ solve the same PDE problem except where firm $i$ is removed and thus there are $N-1$ firms in the market.  This is similar to the duopoly problem where the boundary condition depends on the monopoly problem. The remaining conditions can be worked out in a similar fashion.

\end{appendix}

\bibliographystyle{plainnat}
\small{\bibliography{refs_fin}}
\end{document}